\documentclass[12pt]{amsart}

\usepackage{vmargin}
\setmargrb{1in}{1in}{1in}{1in}
\usepackage{amstext}
\usepackage{amssymb,mathrsfs}
\usepackage{latexsym}

\usepackage{color}

\definecolor{red2}{rgb}{0.933333,0,0}             %
\definecolor{black}{rgb}{0,0,0}                   %
\definecolor{dodgerblue}{rgb}{0.117647,0.564706,1}%
\definecolor{red3}{rgb}{1,0,0.5}                  %
\newcommand\ch{\textcolor{black}}                 %
\newcommand\chhf{\textcolor{black}}               %
\newcommand\cch{\textcolor{black}}           %
\newcommand\RR{{\mathbb R}}

\newcommand\NN{{\mathbb N}}
\newcommand\ZZ{{\mathbb Z}}
\newcommand\TT{{\mathbb T}}

\def\span{{\rm span}}

\def\d={\,:=\,}

\font\frakten=eufm10
\newfam\frakfam
\textfont\frakfam=\frakten

\newtheorem{thm}{Theorem}
\newtheorem{lemma}[thm]{Lemma}
\newtheorem{cor}[thm]{Corollary}
\newtheorem{prop}[thm]{Proposition}
\newtheorem{Defn}[thm]{Definition}
\newtheorem{Ex}[thm]{Example}
\newtheorem{Rem}[thm]{Remark}
\newtheorem{Exs}[thm]{Examples}
\newtheorem{Rems}[thm]{Remarks}
\newtheorem{Defrem}[thm]{Definition and Remark}
\newtheorem{Remnt}[thm]{}
\newenvironment{defn}
 {\begin{Defn} \begin{rm}} {\end{rm} \hfill $\Box$ \end{Defn}}

\newenvironment{ex}
 {\begin{Ex} \begin{rm}} {\end{rm} \hfill $\Box$ \end{Ex}}

\newenvironment{rem}
 {\begin{Rem} \begin{rm}} {\end{rm} \hfill $\Box$ \end{Rem}}

\begin{document}

\author{J.~Bruna, J.~Cuf\'i, H.~F{\"u}hr, M.~Mir\'o}

%\address{Institute of Biomathematics and Biometry \\
%GSF Research Center for Environment and Health \\
%Ingolst{\"a}dter Stra{\ss}e~1 \\
%D-85764 Neuherberg \\
%Germany.}
%\email{fuehr@gsf.de}

\title[Characterizing abelian admissible groups]{Characterizing abelian admissible groups}

\thanks{The research of J. Bruna, J. Cuf\'{\i} and M. Mir\'{o} are partially supported by the grants
MTM2008-05561-C02-02 of the Ministerio de Ciencia e Innovaci\'{o}n
of Spain and 2009 SGR 1303 of the Generalitat de Catalunya}


\date{\today}

\begin{abstract}
By definition, admissible matrix groups are those that give rise to
a wavelet-type inversion formula. This paper investigates necessary
and sufficient admissibility conditions for abelian matrix groups.
\chhf{We start out by deriving a block diagonalization result for
commuting real valued matrices. We then reduce the question of
deciding admissibility to the subclass \cch{of} connected and simply
connected groups, and derive a general admissibility criterion for
exponential solvable matrix groups. For abelian matrix groups with
real spectra, this yields an easily checked necessary and sufficient
characterization of admissibility.  As an application, we sketch a
procedure how to check admissibility of a matrix group generated by
finitely many commuting matrices with positive spectra.}

We also present examples showing that the simple answers that are available for the real spectrum case fail in the general case.

An interesting byproduct of our considerations is a method that allows for an abelian Lie-subalgebra $\mathfrak{h} \subset gl(n,\mathbb{R})$ to check whether $H = \exp(\mathfrak{h})$ is closed.
\end{abstract}

\maketitle

\section{Introduction}  \label{sect:intro}

Let us quickly recall the group-theoretic formalism for the construction of continuous wavelet transforms in higher dimensions. For a more complete introduction, we refer to~\cite{FuMa,LWWW}.
The starting point is a subgroup $H < {\rm GL}(n,\mathbb{R})$, called the {\em dilation group}. Its action on $\mathbb{R}^n$ gives rise to the semidirect product $G = \mathbb{R}^n \rtimes H$, which is just the group of affine mappings generated by $H$ and all translations. We write elements of this group as $(x,h)$, with $x \in \mathbb{R}^n$, $h \in H$. This group acts unitarily on the Hilbert space ${\rm L}^2(\mathbb{R}^n)$ via the {\em quasiregular representation}
\[
 \pi(x,h) f(y) = |{\rm det}(h)|^{-1/2} f(h^{-1}(y-x))~.
\]
The associated continuous wavelet transform of $f \in {\rm L}^2(\mathbb{R})$ is obtained by picking
a suitable $\psi \in {\rm L}^2(\mathbb{R})$ and letting
\[
 V_\psi f (x,h) = \langle f, \pi(x,h) \psi \rangle = \int_{\mathbb{R}^n} f(y) |{\rm det}(h)|^{-1/2} \overline{\psi \left( h^{-1}(y-x) \right)} dy~\mbox{ for } (x,h) \in G~.
\]
The wavelet $\psi$ is called {\em admissible} if $V_\psi : {\rm
L}^2(\mathbb{R}^n) \to {\rm L}^2(\ch{\rm G})$ is isometric. In this
case, we have the {\em wavelet inversion formula}
\[
  f(y) = \int_{\ch{H}} \int_{\ch{\mathbb{R}^n}} V_\psi f(x,h) |{\rm det}(h)|^{-1/2} \psi \left( h^{-1}(y-x) \right)  d{\chhf{x}} \frac{d\ch{h}}{|{\rm det}(h)|} ~,
\] to be read in the weak sense (rather than pointwise)\ch{, where $dh$ is the left Haar measure on~$H$.} The matrix group $H$ is called {\em admissible} if there exists an admissible vector in~${\rm L}^2(\mathbb{R}^n)$.

Necessary and sufficient criteria for admissibility have been studied with increasing generality since the early nineties; see eg. \cite{Mu,Bo,BeTa,Fu,Fu2,FuMa,LWWW}. A complete characterization of admissibility in terms of the dual action appeared quite recently \cite{Fu_Calderon}. The existing literature already suggests a large variety of admissible matrix groups. This paper undertakes a more systematic study of admissibility for the subclass of abelian matrix groups. These groups were previously studied in \cite{Fu2,Larsonetal}, with the former source focussing on the dilation groups $H$ for which $\pi$ is a finite sum of irreducibles, and the latter on one-parameter groups. Already in the restricted setting of \cite{Fu2}, the class of admissible matrix groups quickly becomes too large to manage: It turns out that the conjugacy classes of admissible abelian matrix groups for which the quasiregular representation is irreducible are in natural correspondence to the isomorphism classes of commutative algebras with unity of the same dimension up to isomorphism. In particular, a classification modulo conjugacy of this subclass is out of sight.

Thus it seems that the best one can hope for are methods that allow, for a concrete matrix group described by finite data (e.g.,  by --possibly infinitesimal-- generators), to decide admissibility in an algorithmic way.
Our aim is to derive sufficient conditions that are easy to check and widely applicable. In particular, we want to address finitely generated abelian dilation groups: Given any set of pairwise commuting matrices $A_1,\ldots,A_k$, when does there exist a $\psi \in {\rm L}^2(\mathbb{R}^n)$ guaranteeing the wavelet inversion formula
\[
 f(y) = \sum_{\ell \in \mathbb{Z}^k} \int_{\ch{\mathbb{R}^n}} V_\psi f(x,A^\ell) |{\rm det}(A^\ell)|^{-1/2} \psi \left( A^{-\ell}(y-x) \right) dx ~\chhf{?}
\] Here we used the multi-index notation $A^\ell = A_1^{\ell_1} \cdot \ldots \cdot A_k^{\ell_k}$.
The interest in discrete groups is amplified by the following
observation: Suppose that, for a suitable lattice $\Gamma \subset
\mathbb{R}^n$ the system $(\pi(x,A^\ell) \psi)_{x \in \Gamma, \ch{
\ell \in \mathbb{Z}^k}}$ is a wavelet frame, i.e., one has for all
$f \in {\rm L}^2(\mathbb{R}^n)$ that
\[
 A \| f \|_2^2\le \sum_{x,\ell} \left| \langle f, \ch{\pi}(x,A^\ell) \psi  \ch{\rangle} \right|^2 \le B \| f \|_2^2~.
\] Then the results of \cite{AnCaDeLe} imply that $H = \langle A_1,\ldots,A_k \rangle$ is admissible.
Thus the methods developed here will allow to derive necessary criteria for dilation groups generating a wavelet frame.

\chhf{Let us now give a short outline of the paper and the main results. Section \ref{sect:adm_matrix} contains an overview of admissible matrix groups and the criteria characterizing them. Section \ref{sect:comm_matrix} is concerned with the structure of commuting matrices. Theorem \ref{thm:structure_commuting_real} is the main structure result in this section, describing a form block diagonalization of commuting matrices. It is probably well-known; we include a full proof to stress that the bases corresponding to the block diagonalization can be computed via the Gauss algorithm, if the spectra of the commuting matrices are known. As an application of the structure result, we consider the question whether the exponential image of an abelian matrix Lie algebra is closed. It turns out that once the spectra of a set of generators are known, this question can be decided by repeated applications of the Gauss algorithm (Theorem \ref{thm:char_closed_abel}). This theorem, and some of the auxiliary results leading up to it, can be considered to be of independent interest, but they are also crucial for the subsequent results, in particular for the proof and construction of counterexamples in the final section of this paper. In Section \ref{sect:cont_adm_groups}, we return to the discussion of admissible groups. Here, the chief result is Proposition \ref{prop:struct_ab}, reducing the problem of deciding admissibility for arbitrary abelian matrix groups to the subclass of simply connected, connected groups. For this setting, the admissibility criteria from Section \ref{sect:adm_matrix} are investigated more closely in Section \ref{sect:adm_casc}.  It turns out that for connected abelian matrix groups with real spectrum, there exist simple computational criteria allowing to decide admissibility (Corollary \ref{cor:solv_adm}, in conjunction with Proposition \ref{prop:local_free}). Section \ref{sect:fga_real_spect} applies these results to describe a procedure for checking admissibility of a matrix group generated by finitely many commuting matrices with positive spectra. Finally, Section \ref{sect:counter_ex} gives examples showing that the criteria established for the real spectrum case can fail in the general case.
}

\section{Admissible matrix groups}
\label{sect:adm_matrix}

Throughout the paper, $GL(n,\RR)$ denotes the group of invertible
real-valued $n \times n$ matrices, $gl(n,\RR)$ the vector space of
all matrices. A {\em matrix group} is a subgroup of $GL(n,\RR)$.
For a Lie-subgroup of $H$ of $GL(n,\RR)$, the Lie-algebra is the
subspace ${\mathfrak{h}} \subset gl(n,\RR)$ of tangent vectors of
curves in $H$ through identity, endowed with the usual matrix
commutator.

\ch{The dual group $\widehat{\mathbb{R}^n}$ is the character group
of $\mathbb{R}^n$ suitably identified with the space of row
vectors.} In the following, the Fourier transform of $f \in {\rm
L}^1(\mathbb{R}^n)$ is defined as
\[
 \widehat{f}(\xi) = \int_{\mathbb{R}^n} f(x) e^{- 2 \pi i \ch{\langle \xi,x \rangle}} dx~.
\]
Admissibility of a matrix group $H$ is closely related to its {\em
dual action}. It can be understood as a right action on the Fourier
\ch{side.} Given a dilation group \ch{$H < {\rm GL}(n,\mathbb{R})$
and $\psi\in {\rm L}^2(\mathbb{R}^n) $}, admissibility of $\psi$ is
equivalent to the {\em Calderon condition} \cite{FuMa,LWWW}
\begin{equation} \label{eqn:Calderon}
 \int_H |\widehat{\psi}(\ch{\omega h})|^2 dh = 1 ~~,\mbox{a.e.}~ \omega \in \ch{\widehat{\mathbb{R}^n}} ~~.
\end{equation}
A particularly well-studied case concerns the existence of {\em open
dual orbits}, which is closely related to so-called {\em
discrete-series representations}, see \cite{BeTa,Fu}. In this case,
\ch{$\widehat{\mathbb{R}^n}$} is the union of finitely many orbits
(up to a set of measure zero), and admissible vectors exist iff all
open orbits have associated compact stabilizers \cite{Fu}.

We will now present an overview of known properties of admissible matrix groups. We first recall an observation from \cite{Fu2}:
\begin{lemma}
 Let $H < {\rm GL}(n,\mathbb{R})$ be a Lie subgroup. If it is admissible, it is closed.
\end{lemma}

In the general setting, the following characterization of admissible
matrix groups can be derived. For the formulation, \ch{given $\omega
\in \widehat{\mathbb{R}^n}$} let ${\rm stab}_H(\omega)$ denote the
stabilizer \ch{of $\omega$} under the dual action of \ch{$H$ in
$\widehat{\mathbb{R}^n}$}. Given a Borel $H$-invariant subset $U
\subset \widehat{\mathbb{R}^n}$, a subset $C \subset U$ meeting each
$H$-orbit in precisely one point is called a {\em fundamental
domain}. We call a subset $C \subset \widehat{\mathbb{R}^n}$ a {\em
measurable fundamental domain} if $C$ is a Borel subset and a
fundamental domain for an $H$-invariant Borel subset of full
measure. The following theorem is \cite[Theorem 6]{Fu_Calderon}; see
\cite{MiroThesis} for further equivalent formulations of the
criterion.

\begin{thm} \label{thm:char_adm}
Let $H < {\rm GL}(n,\mathbb{R})$ be given. Then the following are equivalent:
\begin{enumerate}
 \item[(a)] $H$ is admissible.
 \item[(b)] The following three conditions are fulfilled:
\begin{enumerate}
\item[(i)] There exists a measurable fundamental domain of $\widehat{\mathbb{R}^n}$.
 \item[(ii)] For almost all $\omega \in \widehat{\mathbb{R}^n}$, ${\rm stab}_H(\omega)$ is compact.
 \item[(iii)] There exists $h \in H$ such that $|{\rm det}(h)| \not= \Delta_H(h)$.
\end{enumerate}
\end{enumerate}
\end{thm}
Here $\Delta_H$ denotes the modular function of $H$. For abelian groups $H$, note that $\Delta_H(h)\equiv 1$.

Regarding the compact stabilizer condition, let us note the following facts:
\begin{lemma} \label{lem:cp_stab}
 Let $H$ be a closed matrix group. \cch{Then:}
\begin{enumerate}
 \item[(a)] The set $\Omega_c = \{ \omega \in \ch{\widehat{\mathbb{R}^n}
} : \ch{{\rm stab}_H(\omega)} \mbox{ is compact } \}$ is a Borel
subset of $\widehat{\mathbb{R}^n}$.
 \item[(b)] Suppose that $H$ is discrete. Then $\Omega_c$ is conull.
\item[(c)] Suppose that $H$ is simply connected, connected and abelian. Then $H$ acts freely on $\Omega_c$.
\end{enumerate}
\end{lemma}
\begin{proof}
For part (a), we refer to \cite[5.6]{Fu_LN}. For part (b), note that
$h \in \cch{{\rm stab}_H(\omega)}$ if and only if $\omega$ belongs
to the eigenspace of \ch{$h$} associated to the eigenvalue $1$. For
$h \not= \mathbf{1}$, this eigenspace is proper, and the union over
all associated eigenspaces is a set of measure zero. For $\omega$
outside this set, the stabilizer is trivial. For part (c), note that
$H \cong \mathbb{R}^k$, and thus $H$ has no compact subgroups.
\end{proof}

\section{Structure of commuting matrices}
\label{sect:comm_matrix}

It is well-known that, given a set of commuting matrices over the complex numbers, there exists a basis with respect
 to which all matrices have upper triangular form. In this section, we will derive a similar result over the reals.
Here, upper tridiagonalization generally cannot be fully achieved. However, an approach that is similar to the derivation of the real Jordan
form allows to \chhf{formulate} a useful replacement.

The formulation of the result requires additional terminology. We
write $\mathbb{K}$ for an element of $\{ \mathbb{R},\mathbb{C} \}$.
We will often use the natural embedding $gl(n,\mathbb{R}) \subset
gl(n,\mathbb{C})$. On the other hand, it will be convenient to
identify complex-valued matrices with real-valued ones with double
dimensions: Given a matrix $A \in gl(n,\mathbb{C}), A =
(a_{i,j})_{i,j}$, we denote by $i_{\mathbb{C}}(A)$ the $2n \times
2n$-matrix obtained by replacing each complex entry $a_{i,j}$ by the
real matrix $\left( \begin{array}{cc} {\rm Re}(a_{i,j}) & -{\rm
Im}(a_{i,j}) \\ {\rm Im}(a_{i,j}) & {\rm Re}(a_{i,j}\ch{)}
\end{array} \right)$. Thus $i_{\mathbb{C}} : gl(n,\mathbb{C}) \to
gl(2n,\mathbb{R})$ is an $\mathbb{R}$-algebra monomorphism, and we
identify $gl(n,\mathbb{C})$ with its image. Furthermore, we let
$\mathcal{N}(n,\mathbb{K})$ denote the subspace of proper upper
triangular matrices over $\mathbb{K}$.

Given a matrix $A \in gl(n,\mathbb{C})$, we let ${\rm spec}(A)$ denote its set of eigenvalues.

The following structure result will be useful for the study of properties of abelian matrix groups.
\begin{thm} \label{thm:structure_commuting_real}
Let $A_1,\ldots,A_k \in gl(n,\mathbb{R})$ be commuting matrices. Then there exists $B \in GL(n,\mathbb{R})$, $d_r \in \mathbb{N}$ and $\mathbb{K}_r \in \{ \mathbb{R},\mathbb{C} \}$ (for $r=1,\ldots,\ell$) such that
\[
  \sum_{r=1}^\ell d_r \cdot {\rm dim}_{\mathbb{R}} \mathbb{K}_r = n
\]
and, for $j=1,\ldots,k$,
\begin{equation}
 \label{eqn:blockst_real}
B A_j B^{-1} = \left( \begin{array}{cccc} A_{j,1} & 0 & \ldots & 0 \\ 0 &  A_{j,2} & 0
\ldots & 0 \\ 0 & 0 & \ddots & 0 \\ 0 & \ldots & \ldots & A_{j,\ell} \end{array}
\right) ~~
\end{equation} with blocks
\[ A_{j,r} \in  \mathbb{K}_r \cdot \mathbf{1}_{d_r} + \mathcal{N}(d_r,\mathbb{K}_r) ~. \]
If the spectra of $A_1,\ldots,A_k$ are known, $B$ is explicitly computable by repeated applications of Gauss elimination. \\
One has $\mathbb{K}_1 = \ldots = \mathbb{K}_\ell = \mathbb{R}$ iff ${\rm spec}(A_r) \subset \mathbb{R}$ for all $1 \le r \le \ell$.
\end{thm}

The existence of $B$ can be concluded from the structure results for maximal abelian matrix algebras, as developed in \cite{SuTy}. However, the explicit calculation of $B$ is not addressed in that source; to begin with, we would have to compute a maximal abelian matrix algebra containing the $A_j$. For this reason, we give a full proof. The proof strategy consists in first computing a decomposition of $\mathbb{C}^n$ into a sum of invariant subspaces $V_i$ with the property, that the restriction of each $A_j$ has a single complex eigenvalue. We then compute bases of the $V_i$ that triangularize these restrictions. Real-valued bases are constructed by combining bases of suitable pairs of $V_i$. The union of these bases then gives the columns of $B$.

\begin{defn}
 Let $A \in gl(n,\mathbb{K})$,  $\lambda \in \mathbb{K}$ and a subspace $V \subset \mathbb{K}^n$ be given. We write
\[
N(A,\lambda,V) = \{ v \in V:  (A-\lambda\ch{\mathbf{1}})^n (v) = 0
\} ~.\] We let $N(A,\lambda) = N(A,\lambda,\mathbb{K}^n)$. Given
tuples $\mathbf{A}=(A_1,\ldots,A_k)$ of matrices and
$\mathbf{\lambda} \in \ch{\mathbb{K}^{k}}$, we \chhf{define}
\[
 N(\mathbf{A},\mathbf{\lambda}) = \bigcap_{j=1}^k N(A_j,\lambda_j)~.
\]
\vspace*{-7pt}
\end{defn}

\pagebreak

\begin{lemma} \label{lem:decomp_gen_es}
 Let $A \in gl(n,\mathbb{K})$. Suppose that $V$ is an $A$-invariant subspace and let $V = \bigoplus_i V_i$ be a direct sum decomposition into $A$-invariant subspaces. Then
\[
 N(A,\lambda,V) = \bigoplus_i N(A,\lambda,V_i)~.
\]
\end{lemma}
\begin{proof}
 The inclusion $\supset$ is clear. For the other direction, suppose that $v \in N(A,\lambda,V)$, and $v = \sum_{i} v_i$. Then
\[
 0 = (A-\lambda \ch{\mathbf{1}})^n (v) = \sum_{i} (A-\lambda \ch{\mathbf{1}})^n(v_i)
\] with $(A-\lambda \ch{\mathbf{1}})^n (v_i) \in V_i$, since $V_i$ is invariant. Now directness of the sum implies that $v_i \in N(A,\lambda,V_i)$.
\end{proof}

\begin{lemma} \label{lem:eigenspaces}
Let $A_1,\ldots,A_k \in gl(n,\mathbb{C})$ be pairwise commuting \chhf{with known spectra.}
\begin{enumerate}
 \item[(a)] For any given matrix $B$ commuting with $A_1,\ldots,A_k$ and $\mathbf{\lambda} \in \mathbb{C}^k$, the space $N(\mathbf{A},\mathbf{\lambda})$ is $B$-invariant.
 \item[(b)] Let $\Lambda_{\mathbf{A}} =\prod_{j=1}^k {\rm spec}(A_j)$. Then
\[
 \mathbb{C}^n = \bigoplus_{\mathbf{\lambda} \in \Lambda_{\mathbf{A}}} N(\mathbf{A},\mathbf{\lambda})~,
\] where some of the spaces on the right hand side may be trivial.
\item[(c)] For each $\mathbf{\lambda} \in \chhf{\Lambda_{\mathbf{A}}}$, a basis of $ N(\mathbf{A},\mathbf{\lambda})$ is computable via repeated Gauss elimination steps.
 \item[(d)] If $\mathbf{\lambda} \in \cch{\mathbb{R}^k} \cap \Lambda_{\mathbf{A}}$ and $A_1,\ldots,A_k \in gl(n,\mathbb{R})$, \chhf{ then the basis from (c) can be computed with real entries.}
 \item[(e)] If $A_1,\ldots,A_k \in gl(n,\mathbb{R})$, the mapping $\mathbb{C}^n \ni v \mapsto \overline{v}$ (componentwise complex conjugation) induces a bijection
$N(\mathbf{A},\mathbf{\lambda}) \to N(\mathbf{A},\overline{\mathbf{\lambda}})$.
\end{enumerate}
\end{lemma}

\begin{proof}
It is easily seen that $N(A,\lambda)$ is $B$-invariant if $[A,B]=0$. This implies (a), since the intersection of invariant spaces is invariant.

The proof of part (b) is a straightforward induction argument. The
base case $k=1$ is the first step in the Jordan decomposition of
$A_1$. In the induction step, let $A_1,\ldots,A_{k+1}$ be given. Let
$\ch{\mathbf{A}} = (A_1,\ldots,A_{k+1})$ and $\mathbf{A}' =
(A_1,\ldots,A_k)$ . The induction hypothesis yields
\[
 \mathbf{C}^n = \bigoplus_{\mathbf{\lambda}' \in \Lambda_{\mathbf{A}'}} N(\mathbf{A}',\ch{\mathbf{\lambda}'})~.
\] Note that each subspace on the right-hand side is $A_{k+1}$-invariant. Lemma \ref{lem:decomp_gen_es} implies for $\lambda \in {\rm spec}(A_{k+1})$ that
\begin{eqnarray*}
 N(A_{k+1},\lambda) & = & \bigoplus_{\mathbf{\lambda}' \in \Lambda_{\mathbf{A}'}} N(\mathbf{A}',\ch{\mathbf{\lambda}'}) \cap N(\ch{A}_{k+1},\lambda) \\
 & = & \bigoplus_{\mathbf{\lambda} \in \Lambda_{\mathbf{A}}, \lambda_{k+1}=\lambda} N(\mathbf{A},\mathbf{\lambda}).
\end{eqnarray*} Now taking the sum over $\lambda \in {\rm spec}(A_{k+1})$ yields the desired decomposition.

The bases in (c) are best computed simultaneously, and again by
induction over $k$. We start out by computing bases of the
eigenspaces $N(A_1,\lambda)$, for $\lambda \in {\rm spec}(A_1)$.
Recall that the Gauss algorithm can be employed to compute a basis
of a kernel of a given matrix. This settles the case $k=1$. In the
induction step, note that the computations so far provide bases of
all $ N(\mathbf{A}',\ch{\mathbf{\lambda}'})$, which are
$A_{k+1}$-invariant. Therefore, $A_{k+1}$ already has block diagonal
form with respect to the basis (and the blocks are computable), and
it remains to compute bases for the generalized eigenspaces of the
blocks.

Since the kernel of a real-valued matrix (viewed as an operator on $\mathbb{C}^n$) has a real-valued basis, we find that the proof of (c) yields part (d) as a byproduct.

Finally, part (e) follows from the simple observation that
\[
 (A_i - \overline{\lambda}_i \ch{\mathbf{1}})^n (\overline{v}) = 0 \Leftrightarrow \overline{(\cch{A_{i}}-\lambda_i \ch{\mathbf{1}})^n(v)} = 0~.
\]
\end{proof}

The following lemma recalls the well-known fact that pairwise
commuting matrices can be jointly upper triangularized over the
complex numbers. See, for instance, \cite[Theorem~3.7.3]{Var} for
the more general case of solvable Lie algebras. The proof therefore
emphasizes  the explicit computability of the matrix $B$ using only
standard linear algebra tools.
\begin{lemma} \label{lem:triangularize}
 Let $A_1,\ldots,A_k \in gl(n,\mathbb{C})$ be pairwise commuting. There exists an explicitly computable $B \in GL(n,\mathbb{C})$ such that, for $j=1,\ldots,k$, the matrix $BA_jB^{-1}$ is upper triangular. If ${\rm spec}(A_j) \subset \mathbb{R}$, for all $j=1,\ldots,k$, the matrix $B$ can be chosen in $GL(n,\mathbb{R})$.
\end{lemma}
\begin{proof}

\ch{One first proves by induction over $k$ that $A_1,\ldots,A_k$
share a nontrivial eigenvector which is explicitly computable. This
being trivial for $k=1$, assume that $v$ is a eigenvector of
$A_1,\ldots, A_{k-1}$ with  eigenvalues $\lambda_1, \ldots,
\lambda_{k-1}$, respectively, meaning that the subspace $E=\{u: A_j
u=\lambda_j u, \,j=1,\ldots, k-1 \}$ is non-trivial. Now, $E$ is
easily seen to be invariant by $A_k$, because $[A_i,A_j]=0$, whence
$A_k$ restricted to $E$ has a non-trivial eigenvector.}

 \ch{We prove now the lemma  by induction
over $n$; the case $n=1$ being trivial. Let $b_n$ denote a common
eigenvector of  $A_1, \ldots, A_k$} and compute vectors
$v_1,\ldots,v_{n-1}$ such that $v_1,\ldots,v_{n-1},b_n$ is a basis.
Writing these vectors into a matrix yields an invertible matrix~$B_0$ such that
\[
 B_0 A_j B_0^{-1} = \left( \begin{array}{cc} \tilde{A}_j & y_j \\ 0 & \lambda_j \end{array} \right) ~,
\] with $\tilde{A}_j \in gl(n-1,\mathbb{C})$ \ch{and $y_{j} \in \mathbb{C}^{n-1}$}. The $\tilde{A}_j$ are pairwise commutative, hence the induction hypothesis yields an explicitly computable matrix $C \in GL(n-1,\mathbb{C})$ such that $C \tilde{A}_j C^{-1}$ is upper triangular. But then it is easy to see that the matrix $B =  \left( \begin{array}{cc} C & 0 \\ 0 & 1 \end{array} \right) B_0$ is as desired.

In the real-spectrum case, all of the steps in the computation of $B$ can be carried out over the reals.
\end{proof}

\begin{proof}[Proof of Theorem \ref{thm:structure_commuting_real}]
We first compute bases of the subspaces in the decomposition
\[
 \mathbb{C}^n = \bigoplus_{\mathbf{\lambda} \in \Lambda_{\mathbf{A}}} N(\mathbf{A},\mathbf{\lambda})~
\] \ch{of Lemma 7 (b).} Since the spaces are $A_j$-invariant, for all $j$, it follows that $A_j$ has block diagonal form, and the blocks of $A_i,A_j$ associated to the same subspace commute. Furthermore, Lemma \ref{lem:eigenspaces} (d) allows to choose real-valued bases whenever $\mathbf{\lambda} \in \mathbb{R}^k$, and then the corresponding blocks are real-valued as well.

We let $\Lambda_{\mathbf{A}}^0$ denote the set of \ch{$\lambda$} for
which $N(\mathbf{A},\mathbf{\lambda})$ is nontrivial. Pick a subset
$\Lambda' \subset \Lambda_{\mathbf{A}}^0$ containing precisely one
of $\mathbf{\lambda}, \overline{\mathbf{\lambda}}$, for each nonreal
$\mathbf{\lambda} \in \Lambda_{\mathbf{A}}^0$. It follows that
\[
 \mathbb{C}^n = \left( \bigoplus_{\mathbf{\lambda} \in \Lambda_{\mathbf{A}} \cap \mathbb{R}^k} N(\mathbf{A},\mathbf{\lambda})  \right) \oplus \left( \bigoplus_{\mathbf{\lambda} \in \Lambda'} N(\mathbf{A},\mathbf{\lambda})  \oplus N(\mathbf{A},\overline{\mathbf{\lambda}}) \right)~.
\]
For $\mathbf{\lambda} \in \Lambda_{\mathbf{A}} \cap \mathbb{R}^k$, we compute a real-valued basis triangularizing the blocks of the $A_j$. Since each block has a single eigenvalue (this is the whole point of using the space $N(\mathbf{A},\mathbf{\lambda})$, the triangularized blocks are therefore in  $\mathbb{R} \cdot \mathbf{1}_{d} + \mathcal{N}(d,\mathbb{R})$, with $d$ being the dimension of~$N(\mathbf{A},\mathbf{\lambda})$.

Finally, consider $\ch{\lambda} \in \ch{\Lambda'}$. Denote the block
over $ N(\mathbf{A},\mathbf{\lambda})$  associated to $A_j$ by $B_j$
($j=1,\ldots,k$). Using Lemma \ref{lem:triangularize}, we can
compute a basis $v_1,\ldots,v_d$ of $N(\mathbf{A},\mathbf{\lambda})$
triangularizing the $B_j$. Hence there exist upper triangular
matrices $C_j = (c^j_{s,r})_{s,r=1,\ldots,d}$ such that  $\ch{B_j}
v_s = \sum_{r=1}^d c^j_{s,r}v_r$, for $s=1,\ldots,d$. Since the
$B_j$ all have a single eigenvalue, the triangular matrices are in
$\mathbb{C} \cdot \mathbf{1}_{d} + \mathcal{N}(d,\mathbb{C})$.

Let $V \!=\! N\!(\mathbf{A},\mathbf{\lambda}) \oplus N\!(\mathbf{A},\overline{\mathbf{\lambda}})$.  Our next aim is to show that ${\rm Re}(v_1),{\rm Im}(v_1),\ldots,{\rm Re}(v_d), {\rm Im}(v_d)\!$ is a basis of~$V$. For this purpose, first observe that
 $v \mapsto \overline{v}$ is a conjugate-linear bijective map between $N(\mathbf{A},\mathbf{\lambda})$ and
$N(\mathbf{A},\overline{\mathbf{\lambda}})$, which implies that
$v_1,\ldots,v_d,\overline{v}_1,\ldots,\overline{v}_d$ is a basis of~$V$. But this easily implies that $({\rm Re}(v_1),{\rm
Im}(v_1),\ldots,{\rm Re}(v_d), {\rm Im}(v_d))$ is a basis as well,
and in addition real-valued. Now the fact that \ch{$B_j$} is
real-valued yields that
\begin{eqnarray*}
 \ch{B_j} {\rm Re}(v_s) = {\rm Re} \ch{(B_j v_s)} & = & \sum_{r=1}^d {\rm Re}(c^j_{s,r} v_r) \\
 & = & \sum_{r=1}^d {\rm Re}(c^j_{s,r}) {\rm Re}(v_r) - {\rm Im}(c^j_{s,r}) {\rm Im}(v_r)~,
\end{eqnarray*}
and similarly
\begin{eqnarray*}
 \ch{B_j} {\rm Im}(v_s) = \sum_{r=1}^d {\rm Im}(c^j_{s,r}) {\rm Re}(v_r) + {\rm Re}(c^j_{s,r}) {\rm Im}(v_r)~.
\end{eqnarray*}
But that means that the matrix describing the restriction of
\ch{$B_j$} with respect to the basis $({\rm Re}(v_1),{\rm
Im}(v_1),\ldots,{\rm Re}(v_d), {\rm Im}(v_d))$ is given by
$i_{\mathbb{C}}(C_j)$.

Thus, taking the \chhf{properly indexed} union of the bases constructed for each block yields a basis as postulated in Theorem \ref{thm:structure_commuting_real}.
\end{proof}

We now turn to an application of Theorem \ref{thm:structure_commuting_real}. Recall that a necessary condition for admissibility of a matrix group is that it is closed. In the following, we will consider matrix groups of the form $H = {\rm exp}(\mathfrak{h})$, where $\mathfrak{h} \subset gl(n,\mathbb{R})$ is an abelian Lie-subalgebra. Such subalgebras are constructed by simply picking any set $A_1,\ldots,A_k$ of pairwise commuting, linearly independent matrices and letting $\mathfrak{h} = {\rm span}(A_1,\ldots,A_k)$.
By construction, $H = \exp(\mathfrak{h})$ is a Lie-subgroup, and $\exp$ is a group epimorphism, when we consider $\mathfrak{h}$ with its additive group structure. However, it is not easy to decide whether $H$ will be closed and/or simply connected. The remainder of this section is devoted to proving that these questions can be decided in a computational way, using the decomposition in Theorem \ref{thm:structure_commuting_real}.

We start out with a result \ch{setting} closedness for the real
spectrum case.
\begin{lemma} \label{lem:real_spect_closed}
 Let $\mathfrak{h} \subset gl(n,\mathbb{R})$ be an abelian Lie subalgebra with the property that ${\rm spec}(X) \subset \mathbb{R}$ for all
 $X \in \mathfrak{h}$. Let $H = \exp(\mathfrak{h})$ be the exponential image. Then $H$ is a closed subgroup \ch{of $GL(n,\mathbb{R})$}, and $\exp: \mathfrak{h} \to H$ is a diffeomorphism.
\end{lemma}
\begin{proof}
Let $A_1,\ldots,A_k$ denote a basis of $\mathfrak{h}$. After choosing the right coordinates, we may assume (\ref{eqn:blockst_real}) with $B$ being the identity matrix. Since all $A_i$ have real spectrum, it follows that $\mathbb{K}_r = \mathbb{R}$, for all $r$.

Now consider the Lie algebra \ch{$\mathfrak{g}$} generated by
$A_1,\ldots,A_k$ and the block diagonal matrices with scalar
multiples of the identity on the diagonal (with block sizes matching
those of $A_1,\ldots,A_k$). By construction, $\mathfrak{g} =
\mathfrak{d} + \mathfrak{n}$, where $\mathfrak{d}$ consists of
diagonal matrices and $\mathfrak{n}$~consists of proper upper
triangular matrices. Now the matrix exponential is  a diffeomorphism
of $\mathfrak{d}$ and $\mathfrak{n}$ onto closed matrix groups $D$
and $N$, respectively (this is clear for~$\mathfrak{d}$; for
$\mathfrak{n}$, see~\cite[I.2.7]{HiNe}). \ch{By the assumptions on
$D$ and $N$, any element of $D$ commutes with any element of $N$.
Also, there is a \chhf{smooth} inverse of the embedding $D \to DN$,
simply by setting the off-diagonal \chhf{entries} to zero. These facts imply
that the canonical map $D \times N \to DN$ is a diffeomorphism, that
$DN$ is closed and
 that $\exp$ maps
$\mathfrak{g}$ diffeomorphically onto $DN$.}

But then $\mathfrak{h} \subset \mathfrak{g}$ is mapped diffeomorphically onto $H$, and $H$ is closed as the image of the closed subspace $\mathfrak{h}$.
\end{proof}

As it turns out, the closedness of an abelian matrix group will depend on the imaginary parts on the diagonal. The precise formulation of the conditions require some additional terminology, which is provided by
the following lemma.
\begin{lemma} \label{lem:def_h0}
Let $\mathfrak{h}= \span(A_1,\ldots,A_k) \subset gl(n,\mathbb{R})$ be commutative. After an explicitly computable change of coordinates, there exist $\ell$ and $d_r, \mathbb{K}_r$, $r=1,\ldots \ell$  such that
\[
\mathfrak{h} \subset \mathcal{A} = \left\{ \left( \begin{array}{cccc} B_{1} & 0 & \ldots & 0 \\ 0 &  B_{2} & 0
\ldots & 0 \\ 0 & 0 & \ddots & 0 \\ 0 & \ldots & \ldots & B_{\ell} \end{array}
\right) : B_r \in \mathbb{K}_r \cdot \mathbf{1}_{d_r} + \mathcal{N}(d_r,\mathbb{K}_r) \right\}~.
\]
Then $\mathcal{A}$ is an associative subalgebra of $gl(n,\mathbb{R})$.
 Denote by $P_{\mathcal{E}}: \mathcal{A} \to \mathcal{A}$ the map that discards the imaginary parts on the diagonal, i.e.,
\begin{eqnarray*}
\lefteqn{ P_{\mathcal{E}} :  \left( \begin{array}{cccc} \alpha_1 \cdot \mathbf{1}_{d_1} + N_1 & 0 & \ldots & 0 \\ 0 &  \alpha_2 \cdot \mathbf{1}_{d_2} + N_2 & 0
\ldots & 0 \\ 0 & 0 & \ddots & 0 \\ 0 & \ldots & \ldots & \alpha_\ell \cdot \mathbf{1}_{d_\ell} + N_\ell \end{array}
\right) } \\ &  \mapsto &   \left( \begin{array}{cccc} {\rm Re}(\alpha_1) \cdot \mathbf{1}_{d_1} + N_1 & 0 & \ldots & 0 \\ 0 &  {\rm Re} (\alpha_2) \cdot \mathbf{1}_{d_2} + N_2 & 0
\ldots & 0 \\ 0 & 0 & \ddots & 0 \\ 0 & \ldots & \ldots & {\rm Re}(\alpha_\ell) \cdot \mathbf{1}_{d_\ell} + N_\ell \end{array}
\right)  ~.\end{eqnarray*}
Let $\mathcal{E}$ denote the range of $P_{\mathcal{E}}$. Furthermore, let $P_{\mathcal{I}} = {\rm Id}_{\mathcal{A}}-P_{\mathcal{E}}$, and denote its range by~$\mathcal{I}$.
Then $\mathfrak{h}_0 = \mathfrak{h} \cap \mathcal{I}$ is explicitly computable, and there exists an explicitly computable complement $\mathfrak{h}_1$ of $\mathfrak{h}_0$ in $\mathfrak{h}$. On $\mathfrak{h}_1$, the map $P_{\mathcal{E}}$ is injective.
\end{lemma}
\begin{proof}
 $\mathfrak{h}_0$ is the kernel of $P_{\mathcal{E}}|_{\mathfrak{h}}$, and thus one can compute a basis of this space, together with a basis of the complement, using the Gauss algorithm. Injectivity of $P_{\mathcal{E}}$ on the complement is clear.
\end{proof}

For the formulation of the next lemma, recall that an element $g$ of
a topological group~$G$ is called {\em a compact element} if the
closed subgroup generated by $g$ is \ch{compact}. Given a sequence
$(g_k)_{k\in \mathbb{N}} \subset G$, the statement $g_k \to \infty$
for $k \to \infty$ means that every compact subset $K \subset G$
contains at most finitely many $g_k$.
\begin{lemma} \label{lem:H0_H1}
 Let $\mathfrak{h} \subset gl(n,\mathbb{R})$ be commutative, and let $\mathfrak{h} = \mathfrak{h}_0 \oplus \mathfrak{h}_1$ denote the decomposition from Lemma \ref{lem:def_h0}. Let $H_i = \exp(\mathfrak{h}_i)$.
\begin{enumerate}
\item[(a)]  $H_1$ is closed and simply connected.
\item[(b)]  $h \in H$ is a compact element of $GL(n,\mathbb{R})$ if and only if $h \in H_0$.
\end{enumerate}

\end{lemma}
\begin{proof}
First of all, note that $\exp(\mathcal{I}) \subset SO(n)$, \ch{the
compact special orthogonal group}, and thus in particular
$\exp(\mathfrak{h}_0) \subset SO(n)$, which shows one direction of
part (b), and will be crucial for part (a).

For the proof of part (a), we first prove that whenever $X_k \to
\infty$ in $\mathfrak{h}_1$, then $\exp(X_k) \to \infty$ in
$GL(n,\mathbb{R})$. For this purpose, let $\mathfrak{h}_2 =
P_{\mathcal{E}}(\mathfrak{h}_1)$ and $H_2 = \exp(\mathfrak{h}_2)$.
Then $X_k = Y_k + Z_k$, with $Y_k \in \mathcal{I}$ and $Z_k =
P_{\mathcal{E}}(X_k) \in \mathfrak{h}_2$.  By injectivity of
\ch{$P_{\mathcal{E}}$} on $\mathfrak{h}_1$, it follows that $Z_k \to
\infty$. By Lemma \ref{lem:real_spect_closed}, $\exp: \mathfrak{h}_2
\to H_2$ is a diffeomorphism onto the closed subgroup $H_2$, and
thus $\exp(Z_k) \to \infty$ in $GL(n,\mathbb{R})$. Hence, if $K
\subset GL(n,\mathbb{R})$ is compact, then so is $K' = SO(n) K$, and
it can contain only finitely many $\exp(Z_k)$. On the other hand,
$\exp(X_k) \in K$ implies $\exp(Z_k) = \exp(-Y_k) \exp(X_k) \in K'$,
thus $K$ contains only finitely many $\exp(X_k)$. Hence $\exp(X_k)
\to \infty$. This implies in particular that the kernel of $\exp$
must be compact, hence trivial. \ch{So $H_1$ is simply connected.}

Furthermore, it follows that $H_1$ is closed: Assume that $\exp(X_k) \to g \in GL(n,\mathbb{R})$. This implies that the sequence $(X_k)_{k \in \mathbb{N}}$ does not converge to infinity, hence it contains a bounded, and thus finally a convergent subsequence $X_{n_k} \to X_0$, and $X_0 \in \mathfrak{h}_1$ since linear subspaces are closed. But then continuity of $\exp$ yields
$\exp(X_0) = g$, and thus $g \in H_1$.

 For the missing direction of part (b), write $h = \exp(X_0+X_1)$, with $X_i \in \mathfrak{h}_i$. If $X_1 \not=0$, the argument proving part (a) shows that $h^k = \exp(k(X_0+X_1)) \to \infty$. In particular, $h$~is not a compact element.
\end{proof}

The following theorem reveals the chief purpose of the introduction of $\mathfrak{h}_0,\mathfrak{h}_1$: Closedness only depends on $\mathfrak{h}_0$.
\begin{thm} \label{thm:char_closed_abel_part}
 Let $\mathfrak{h} \subset gl(n,\mathbb{R})$ be commutative, and let $\mathfrak{h} = \mathfrak{h}_0 \oplus \mathfrak{h}_1$ denote the decomposition from Lemma \ref{lem:def_h0}. Then $H = \exp(\mathfrak{h})$ is closed if and only if $H_0 = \exp(\mathfrak{h}_0)$ is compact.
\end{thm}
\begin{proof}
By Lemma \ref{lem:H0_H1}, $H_1$ is closed. Thus, if $H_0$ is
compact, then $H$ is the product of a closed and a compact subset of
$GL(n,\mathbb{R})$, hence closed. On the other hand, if $H$ is
closed, then its subgroup of compact elements is closed as well, and
compactness of an element relative to $H$ is the same as compactness
relative to $GL(n,\mathbb{R})$. This subgroup coincides with $H_0$
by Lemma \ref{lem:H0_H1}, and it is closed by \cite[Theorem 9.10]{HeRo}. In summary: $H_0$ is closed, hence compact, \ch{since it is
contained in $SO(n)$}.
\end{proof}

Hence, we need methods to identify closed subgroups of the torus group $\mathbb{T}^d = \{ z \in \mathbb{C}^d : |z_1|= |z_2| = \ldots = |z_d|=1 \}$. In the following, we identify the Lie algebra of $\mathbb{T}^d$ with $\mathbb{R}^d$, and $\exp: \mathbb{R}^d \to \mathbb{T}^d$ is given by $\exp(x_1,\ldots,x_d) = (e^{ix_1},\ldots,e^{ix_d})$.
\begin{lemma}
 Let $\varphi  = ( \varphi_1,\ldots,\varphi_d) \in \mathbb{R}^d$, and denote by $H$ the closure of $\exp(\ch{\mathbb{R} \varphi})$. Let
$\mathfrak{h} \subset \mathbb{R}^d$ denote the Lie algebra of $H$,
i.e.
\[
 \mathfrak{h} = \{ x \in \mathbb{R}^d : \exp(x) \in H \}~.
\]
Suppose that $1 \le i_0 \le d$ denotes the smallest index of an irrational entry of $\varphi$. Then $\mathfrak{h}$~contains a vector
whose first nonzero component is at position $i_0$.
\end{lemma}
\begin{proof}
\chhf{Since replacing $\varphi$ by any nonzero scalar multiple yields the same one-parameter group, we may assume that the first $i_0-1$ entries are integers.}
 Denote by $H_0$ the closure of the cyclic subgroup  $ \{\exp(2 \pi k \varphi) : k \in \mathbb{Z} \} \ch{\subset \mathbb{T}^d}$. Then, if
 $p: \mathbb{T}^d \to \mathbb{T}^{i_0}$ denotes the projection onto the first $i_0$ components, \chhf{we claim that
 $p(H_0) = \{ 1 \}^{i_0-1} \times \mathbb{T}$.  Indeed, by choice of $\varphi$, $i_0$ and $H_0$  we have
\begin{equation} \label{eqn:incl_pH0}
 \cch{\{ 1 \}^{i_0-1}} \times \{ e^{2 \pi i k \varphi_{i_0}} : k \in \mathbb{Z} \} \subset p(H_0) \subset  \cch{\{ 1 \}^{i_0-1}} \times \mathbb{T}~.
\end{equation}}%
Furthermore, $p(H_0)$ is the continuous image of a compact set, and thus closed. By choice of $i_0$,
$ \{ e^{2 \pi i k \varphi_{i_0}} : k \in \mathbb{Z} \} \subset
\mathbb{T}$ is dense. Hence the first inclusion of
(\ref{eqn:incl_pH0}) and closedness of~\ch{$p(H_0)$} implies
\chhf{$p(H_0)=\{ 1 \}^{i_0-1} \times \mathbb{T}$.}

In particular, $H_0$ is a closed infinite subgroup of $\mathbb{T}^d$, and therefore it is a Lie subgroup of positive dimension. If $\mathfrak{h}_0$ denotes its Lie algebra, then (\ref{eqn:incl_pH0}) implies that $\mathfrak{h}_0 \subset \{ 0  \}^{i_0-1} \times \mathbb{R}^{d+1-i_0}$. On the other hand,  $\mathfrak{h}_0 \not\subset  \{ 0  \}^{i_0} \times \mathbb{R}^{d-i_0}$, since $p(H_0)$ is nontrivial. But this shows the statement.
\end{proof}

With this lemma, closedness of a Lie-subgroup of $\mathbb{T}^d$ is easily
 decided. We first recall some notions connected to Gauss elimination: Let a family of vectors
 $v_j = (v_j(1),\ldots,v_j(d)) \in \mathbb{R}^d$, $j=1,\ldots,k$ be given. We say that the vectors are in {\em Gauss-Jordan \cch{row} echelon form} if there exist indices $1 \le i_1 < i_2 < \ldots < i_j \le d$ such that the following properties hold, for all $j=1,\ldots,k$:
\[
 v_j(r) = 0 ~,~\mbox{ for } r < i_j~,~ v_j(i_j) = 1~,~\mbox{and}~v_j(i_\ell) = 0~,~ \mbox{ for } \ell \not=j~.
\] For any finite family of vectors, a system of vectors in Gauss-Jordan row echelon form spanning the same space can be computed by first computing the echelon form using Gauss elimination, normalizing the resulting vectors to have unit pivot elements, and using the pivot element of each vector to eliminate the corresponding entries in the other vectors.
\begin{lemma} \label{lemma:char_closed_abel}
 Let $v_1,\ldots, v_k \in \mathbb{R}^d$ be given in Gauss-Jordan row echelon form, and let $\mathfrak{h} = {\rm span}(v_j: j \le k )$.
 Then $H = \exp(\mathfrak{h})$ is \ch{compact} iff $v_j \in \mathbb{Q}^d$, for all $1 \le j \le k$.
\end{lemma}
\begin{proof}
 First assume that $v_j \!\in\! \mathbb{Q}^d$. Then $\exp(2 \pi k v_j) \!=\! (1,\ldots,1)$ for a suitable integer~$k \!>\!0$, showing that $\exp(\mathbb{R} v_j)$ is compact. If this holds for all $j$, then $H$ is compact.

Conversely, assume that $H$ is \ch{compact}, but some $v_j$ has an
irrational entry. Let $\ell$ be the smallest index with $v_j(\ell)
\not\in \mathbb{Q}$. Then, by the previous lemma, $\mathfrak{h}$
contains a vector $\varphi$ whose first nonzero component is at
position $\ell$. On the other hand, since $v_j(\ell) \not=0$,  the
fact that the vectors are in Gauss-Jordan row echelon form implies
that $\ell \not\in \{ i_j : j=1,\ldots,k \}$. But clearly, that is a
contradiction to $\varphi \in \mathfrak{h}$.
\end{proof}

To summarize:
\begin{thm} \label{thm:char_closed_abel}
Let $\mathfrak{h} \subset gl(n,\mathbb{R})$ be an abelian Lie algebra. Assume that for some system of generators of
$\mathfrak{h}$ the spectra are known. Then the question whether $H =\exp(\mathfrak{h})$ is closed can be decided by repeated applications of Gauss elimination.
\end{thm}
\begin{proof}
 First compute the decomposition in Theorem \ref{thm:structure_commuting_real}. From this decomposition, determine the direct sum decomposition $\mathfrak{h} = \mathfrak{h}_0 \oplus \mathfrak{h}_1$ from Lemma \ref{lem:def_h0}. Then compute a basis of  $\mathfrak{h}_0$ in Gauss-Jordan echelon form and check for irrational entries.
\end{proof}

\section{\ch{From discrete to continuous admissible matrix groups}}

\label{sect:cont_adm_groups}

A full classification of admissible abelian matrix groups, say up to
conjugacy, does not seem feasible. As the discussion in \cite{Fu2} shows,
already the problem of classifying, up to conjugacy, all connected abelian matrix groups with
open orbits is equivalent to the classification of commutative algebras with
unity of the same dimension. For the latter problem, no solution is in sight. Moreover, we are interested
also in non-connected, even discrete abelian groups, and here the diversity is even larger.
In view of this fact, it seems reasonable to first restrict attention to suitable
subclasses, which are particularly easy to handle, and then to find ways how to pass
from the subclass to general abelian matrix groups. For the first part, connected groups seem particularly
well suited, since the exponential map allows to systematically translate group-theoretical questions to problems in linear algebra. For the passage to other closed abelian matrix groups, the role of cocompactness will be crucial. The following simple lemma provides the key.
\begin{lemma} \label{lem:cocomp}
 Let $H_0, H_1$ be closed \ch{abelian} matrix groups, with $H_0 \subset H_1$ cocompact, \cch{i.e. $H_{1}/H_{0}$ is compact}. Then $H_0$ is
  admissible iff $H_1$ has the same property.
\end{lemma}

\begin{proof}
 First suppose that $H_0$ is admissible. Let $\psi \in {\rm L}^2(\mathbb{R}^n)$ be an admissible
 vector, and define $G_i = \RR^n \rtimes H_i$. Pick a measurable fundamental domain $C$ of $H_1/H_0$, and endow it with the Haar measure of $H_1/H_0$. Then, with suitable normalizations, we have for $f \in {\rm L}^2(\RR^n)$ that
 \begin{eqnarray*}
  \| V_\psi(f) \|_{{\rm L}^2(G_1)}^2 & = & \int_{H_1} \int_{\ch{\RR^n}} |\langle f, \pi(h,x) \psi \rangle|^2 |{\rm det}(h)|^{-1} dx \frac{dh}{|\det{h}|} \\
  & = & \int_C \int_{H_0} \int_{\ch{\RR^n}} |\langle f, \pi(c h_0,x) \psi \rangle|^2 dx \frac{dh_0}{|\det(h_0)|}  \frac{dc}{|\det{c}|} \\
   & = & \int_{C} \| V_\psi(\pi(c^{-1}) f) \|_{{\rm L}^2(G_0)}^2   \frac{dc}{|\det{c}|} \\
   & = & \| f \|_2^2~ \int_{C}  \frac{dc}{|\det{c}|}~.
 \end{eqnarray*}
  Hence  $\psi$ is $H_1$-admissible up to normalization.

  For the converse, let $\psi_1$ denote an $H_1$-admissible function.
   Using the measurable fundamental domain $C$ from above, we define
  \[
   \widehat{\psi_0} (x) = \left( \int_C |\widehat{\psi}_1(c^{-1}x)|^2 \ch{dc} \right)^{1/2}~~.
  \]
  Then a straightforward computation (using that \ch{$H_{1}$} is commutative) shows that $\psi_0 \in {\rm L}^2(\ch{\RR^n})$, and that it fulfills
  the Calderon condition (\ref{eqn:Calderon}) for $H_0$.
\end{proof}

The lemma motivates the following definition: Given $H_1,H_0 \subset
{\rm GL}(n,\RR)$, we write $H_0 \sim H_1$ iff $H_0 \cap H_1$ is
cocompact in both $H_0$ and $H_1$. Then the previous lemma implies
that this relation is compatible with admissibility: If  $H_0 \sim
H_1$, then $H_0$ is admissible iff $H_1$ is. Furthermore, $\sim$ is
reflexive and symmetric, but not transitive. We therefore introduce
its transitive hull, denoted by $\approx$; \ch{that is, $H_0 \approx
H_1$ means that $H_0, H_1$ can be connected by a chain of groups
related by $\sim$.} This is an equivalence relation which is
compatible with admissibility.

We now employ structure theory of compactly generated LCA
groups to prove the following proposition, which shows that each equivalence class modulo $\approx$ contains a connected, simply connected representative $H_c$. Thus, in principle, the discussion may be restricted to this subclass. Note however that the construction of $H_c$ is not particularly explicit, and thus of somewhat limited use for the discussion of concrete examples.  The subsequent remark
\ref{rem:matrix_log} provides a more direct construction of $H_c$ for the subclass of discrete matrix groups with positive spectrum.
\begin{prop} \label{prop:struct_ab}
Let $H < {\rm GL}(n,\RR)$ be closed and abelian.
\begin{enumerate}
\item[(a)] $H \cong \RR^l \times \ZZ^m \times \TT^k \times F$,
with $l,m,k \in \NN_0$, and $F$ denotes a finite abelian group. The
isomorphism is topological. \item[(b)] There exists $H_c \sim H$
with $H_c \cong \RR^j$, and $j=l+m$. Moreover, there exists $H_d
\subset H \cap H_c$, cocompact in both, with $H_d \cong \ZZ^j$. As a
consequence of $H_c \sim H \sim H_d$, $H$ is admissible iff
\ch{$H_d$} is.
\end{enumerate}
\end{prop}

\begin{proof}
For part (a) confer \cite{Wue}. The result basically follows from
the fact that closed abelian subgroups of $GL(n,\RR)$ are
compactly generated, and a structure theorem for such groups
contained in \cite{HeRo}.

For the construction of \ch{$H_{c}$}, we first get rid of the
compact part of $H$, i.e., we let $H_0$ denote the subgroup
corresponding to $\RR^l \times \ZZ^m$. The remaining problem
consists therefore in suitably embedding the discrete part into a
vector group.

For this purpose we let ${\mathcal A} = {\rm span}(H_0)$, the matrix
algebra generated by $H_0$. Then $H_0 \subset {\mathcal A}^\times$,
where the latter denotes the group of invertible elements in
${\mathcal A}$. \cch{We claim that ${\mathcal A}^\times = {\mathcal
A} \cap GL(n,\mathbb{R})$, whence is a closed subgroup of
$GL(n,\mathbb{R})$ (because the algebra ${\mathcal A}$ is a linear
subspace, hence closed). Here the inclusion $\subset$ is clear. For
the other direction suppose that a matrix $X \in {\mathcal A}$ is
invertible. Then left multiplication with $X$ is an injective linear
map of ${\mathcal A}$ into itself, thus it is also onto. Hence the
(right) inverse of $X$ is also in~${\mathcal A}$. The group
${\mathcal A}^{\times}$ is also almost connected by}
\cite[Proposition 10]{Fu2}. Hence the structure theorem for
compactly generated LCA groups yields ${\mathcal A}^\times \cong
\RR^s \times \TT^t \times F$, with a finite abelian group $F$.
Possibly by replacing $H_0$ by a closed subgroup of finite index, we
may assume that $\pi_F(H_0)$ is trivial, where $\pi_F$ is the
projection map onto $F$. \ch{Indeed, we have $H_0 \subset
\mathcal{A}^\times \cong \mathbb{R}^s \times \cch{\mathbb{T}^t}
\times F$, and $\pi_F$ is a continuous homomorphism. Then
$\pi_F(H_0) \subset F$ is a finite group, which means that the
kernel $K$ of $\pi_F$, restricted to $H_0$, is a closed subgroup of
finite index: The isomorphism theorem states that $\pi_F(H_0) \simeq
H_0/K$. We replace $H_0$ by $K$.}

For the sake of explicitness, we introduce topological isomorphisms
$\phi : \RR^l \times \ZZ^m \to H_0$ and $\psi: {\mathcal
A}^{\ch{\times}} \to \RR^s \times \TT^t \times F$. Let
$\psi_1,\ldots, \psi_{s+t}$ denote the $\RR^s \times \TT^t$-valued
components of~$\psi$. Then \[ \Theta: \RR^l \times \ZZ^m \ni (x,m)
\mapsto (\psi_1(\phi(x,m)), \ldots, \psi_{s+t}(\phi(x,m))) \in \RR^s
\times \TT^t \] is a continuous group monomorphism, and it is
topological onto its image. \cch{This image is $\psi(H_{0})$, whence
closed as $H_{0}$ is closed in ${\mathcal A}^{\times}$.} Our aim is
to extend $\Theta$ to a continuous monomorphism $\tilde{\Theta} :
\mathbb{R}^{l+m} \to \mathbb{R}^s \times \TT^t$.

First observe that already the map $\Theta_0\!: \RR^l\! \times \ZZ^m\! \ni\! (x,m)\! \mapsto\!
(\psi_1(\phi(x,m)), \ldots, \psi_{s}(\phi(x,m))) \!\in\! \RR^s$ is
injective: The kernel of $\Theta_0$ is mapped by $\Theta$ onto a
closed subgroup of $ \{ 0 \} \times \TT^t$. Hence it is compact, and thus trivial.

Since  $\mathbb{T}^t$ is compact, the projection $\RR^s \times \TT^t \to \RR^s $ is a closed mapping,
and thus $\Theta_0$ has closed image also. Hence \cite[Theorem 9.12]{HeRo} applies to yield that
\[ \{ \Theta_0(1,0,\ldots,0),\ldots,\Theta_0(0,\ldots,0,1)\} \subset
\RR^s \]
 is $\RR$-linearly independent, giving rise to a linear
monomorphism $\RR^{l +m} \to \RR^s$, written as a tuple
$(\tilde{\theta}_1,\ldots, \tilde{\theta}_s)$ of homomorphisms
$\RR^{l+m} \to \RR$.

Now pick $z_{i,j} \in \RR$ ($1\le i \le m, \ch{s+1 \le j \le s+t}$)
such that \ch{$\psi_{j}(\phi(\delta_{i}))\in (z_{i,j} + \ZZ^t )$},
where $\delta_i \in \RR^l \times \ZZ^m$ denotes the vector with $1$
as \ch{$l+i$} th entry, and zeros elsewhere. Letting
\[ \tilde{\theta}_j(x,y) = \psi_j\ch{(\phi(x,0))} + (\sum_{i=1}^m y_i z_{i,j} + \ZZ^t)~~, \qquad \ch{j=s+1, \ldots, s+t}\]
for $(x,y) \in \RR^l \times \RR^m$ therefore gives rise to an
extension
\[ \widetilde{\Theta}: \RR^l \times \RR^m \ni (x,y) \mapsto
(\tilde{\theta}_1(x,y), \ldots, \tilde{\theta}_{s+t}(x,y)) \in \RR^s
\times \TT^t \] of $\Theta$\cch{; then} $\tilde{\Theta}(\RR^{l+m})$
is closed in $\RR^s \times \TT^t$, since already the projection onto
the first $s$ components yields the closed subgroup
$\Theta_0(\mathbb{R}^\ell \times \mathbb{Z}^m\ch{)}$.

Now $H_c = (\ch{\psi^{-1}\circ
\widetilde{\Theta}})(\mathbb{R}^{l+m})$ is as desired: \ch{by}
construction of $\widetilde{\Theta}$, $H_0 = (\psi^{-1} \circ
\widetilde{\Theta })(\RR^l \times \ZZ^m)$, with $H_c/H_0 \cong
\mathbb{T}^m$. Finally, $H_d=(\psi^{-1} \circ
\widetilde{\Theta})(\ZZ^{l+m})$ is discrete \ch{and cocompact in
$H_0$
%. Moreover, the discrete subgroup $H_d =
%\widetilde{\Theta}(\ZZ^l \times \ZZ^m)$ is cocompact in $H_0$
and thus in $H$.}
\end{proof}

Both the structure of the group and the geometrical intuition of
the action simplify greatly if we can assume that all matrices in
the group have real eigenvalues.
\begin{defn}
An abelian matrix group $H$ {\bf has real (positive) spectrum}
if all $h \in H$ have only real (positive) eigenvalues.
\end{defn}

\begin{rem} \label{rem:matrix_log}
(a) We note that the group $H_c$ in Proposition \ref{prop:struct_ab} (b) can possibly be explicitly computed:
 Let $B_1,\ldots,B_l$ denote infinitesimal generators of the subgroup of $H$ corresponding
 to $\RR^l$, \ch{$B_{l+1},\ldots, B_{l+m}$ generators of the $\ZZ^m$ part, $B_{l+m+1},\ldots,
 B_{l+m+k}$} infinitesimal generators of the $\TT^k$ part, and $B_{l+k+m+1},
 \ldots, B_{l+k+m+f}$ the elements of the finite group.

 Then these matrices commute: This is clear for any pair of matrices contained in the
 discrete \cch{part. Moreover,} if \ch{$1 \le i \le l$ and $l+1 \le j \le l+m$},
differentiating
 the equality \ch{$[{\rm exp}(r B_i),B_j] = 0$} and evaluating at $r=0$ yields \ch{$[B_i,B_j]=0$}.
 Similar arguments apply to the remaining cases, showing that all matrices in the list $B_1,\ldots,
 B_{l+k+m+f}$ commute, and can therefore be jointly decomposed into block triangular form
 as described in Theorem~\ref{thm:structure_commuting_real}. But this block structure is preserved by
 the exponential map, hence it follows that all elements of $H$ have the same block structure.

Thus the decomposition of Theorem \ref{thm:structure_commuting_real}
is valid for the group \ch{$H$} as well. Now the explicit
construction of $H_c$ depends on the unit group of the matrix
algebra $\mathcal{A} = {\rm span}(H)$, which is computable from the
decomposition of the generators into blocks.

(b) The observation in part (a) simplifies the reasoning in
particular for the case of positive spectrum, \chhf{ to the extent that $H_c$ can be computed explicitly:} Suppose that $H$ has
only positive eigenvalues. \chhf{Here, the logarithms of the
generators can be computed} using the power
series
 \begin{equation}
 \label{eqn:matr_log} {\rm log}(A) = \sum_{k=1}^\infty \frac{(-1)^{k+1}}{k} (A-\mathbf{1})^k~~.
 \end{equation}
 First assume that we only have one block. Then $A$ is of the form $a \cdot (\mathbf{1} + N)$, with $a$ a positive
 real and $N$ nilpotent, and we can derive from this
 \begin{equation}
 \label{eqn:matr_log2} {\rm log}(A) = {\rm log}(a) \cdot \mathbf{1} + \sum_{k=1}^{s-1} \frac{(-1)^{k+1}}{k} N^k~~,
 \end{equation} where $s$ is the block size. While these calculations are
 somewhat informal, it can be shown by direct calculation that by defining ${\rm log}(A)$ as in (\ref{eqn:matr_log2}),
 we indeed have ${\rm exp}({\rm log}(A)) = A$ \cite{HiNe}.
 This procedure is applied blockwise to yield the logarithm for the general case.

\chhf{With the notations from part (a), we let $\mathfrak{h}_c =
\span \left(B_1,\ldots,B_l,\log(B_{l+1}),\ldots, \log(B_{l+m})
\right)$; noting that the logarithms are computed in finitely many
steps. We claim that $H_c = \exp(\mathfrak{h}_c)$ is as desired: It
is closed and simply connected by Lemma \ref{lem:real_spect_closed},
and $\exp: \mathfrak{h}_c \to H_c$ is a diffeomorphism. Furthermore,
the additive quotient group  $\mathfrak{h}_c/\langle
B_1,\ldots,\log(B_{l+m}) \rangle$ is quasicompact, by construction
of $\mathfrak{h}_c$. Then the same is true of the exponential
images. But this implies that \cch{$H_c/H_d$} is quasicompact, thus
compact. }
\end{rem}

\begin{rem}
 Let $H< {\rm GL}(n,\mathbb{R})$ be a \chhf{closed} abelian subgroup with real spectrum. Then $H$ contains a closed cocompact subgroup with positive spectrum. To see this, assume that $\theta: H \to \mathbb{R}^l \times \mathbb{Z}^m \times \mathbb{T}^k \times F$ ($F$ finite) is a topological isomorphism, and let $H_0 = \{ h^2 : h \in H \}$. Then $H_0$ is a subgroup with positive spectrum, with $\theta(H_0) = \mathbb{R}^l \times (2 \mathbb{Z})^{m} \times \mathbb{T}^k \times F'$, where $F'$ is the subgroup of squares in $F$.
Thus $H_0$ is as desired.
\end{rem}

%A further benefit of Lemma \ref{lem:cocomp} is that it allows to construct examples
%of discrete abelian matrix groups without measurable fundamental domains. So far, the only discrete
%group found to be unsuited for the construction of wavelets was ${\rm SL}(2,\ZZ)$ and a related
%group \cite{Fu,LWWW}. By Lemma \ref{lem:cocomp}, the existence of a measurable fundamental domain
%for a discrete group $H$ can be checked by computing stabilizers of a connected group $H_1$
%containing $H$ as cocompact subgroup. But the latter task is routine computation.

\section{Admissibility for closed abelian simply connected
matrix groups}

\label{sect:adm_casc}

\ch{Proposition \ref{prop:struct_ab} leads us to consider the
admissibility of a closed abelian simply connected matrix group. In
this section we will study this question and give some
generalizations at the end. The main \chhf{results are Theorem \ref{thm:adm
sca iff} and its generalization, Theorem \ref{thm:solv_adm}.}}

The following necessary
condition can be immediately derived from \ch{Proposition}
\ref{prop:struct_ab}: If $H \cong \RR^l \times \ZZ^m \times \TT^k
\times F$, and $H \sim H_c$, with $H_c \cong \RR^{l+m}$, then a
necessary condition for~$H$ to be admissible is that $H_c$ acts
freely almost everywhere. The reason is that $H_c$ is admissible,
and thus almost every stabilizer in $H_c$ is compact, hence trivial.

Since one motivation for introducing connected groups to the
discussion is their accessibility via Lie algebras, it is therefore
natural to consider Lie algebra criteria for the property that $H_c$
acts freely. Note that this usually only contains information about
{\em local} mapping properties of the group action, hence we can at
best obtain a criterion for the stabilizers to be discrete. This is
easily seen to be equivalent to {\em local freeness} of the action:
$H$ acts locally freely on $H.x$ if there exists a neighborhood $U$
of unity in $H$ such that $h \mapsto h.x$ is injective on $U$. The
following proposition studies the characterization of locally free
actions.
\begin{prop} \label{prop:local_free}
 Let $H < GL(n,\RR)$ be a Lie subgroup of dimension $l$, with Lie algebra ${\mathfrak{h}} \subset gl(n,\RR)$.
\begin{enumerate}
\item[(a)] For $x \in \RR^n$, the stabilizer ${\rm stab}_H(x)$ is discrete iff
\begin{equation} \label{eqn:rank_cond}
 \mbox{The linear map } {\mathfrak{h}} \ni A \mapsto Ax \in \RR^n \mbox{ has rank }
 l\,.
\end{equation}
Hence, necessarily $l \le n$.
\item[(b)] If ${\rm stab}_H(x)$ is discrete for one $x \in \RR^n$, then there exists
an open $H$-invariant set $U \subset \RR^n$ such that $|\RR^n
\setminus U|=0$, and ${\rm stab}_H(y)$ is discrete for all $y \in
U$.
\item[(c)] Given a basis of ${\mathfrak{h}}$, the existence of $x \in \RR^n$ with discrete stabilizer
can be checked computationally.
\end{enumerate}
\end{prop}
\begin{proof}
The stabilizer is discrete iff the canonical mapping $H \ni h
\mapsto hx$ is a local homeomorphism onto the orbit $H.x$. But the
mapping from (\ref{eqn:rank_cond}) is the derivative of this map at
unity. This proves part $(a)$.

For part (b) we note that (\ref{eqn:rank_cond}) is fulfilled iff
there exists an index set $I \subset \{ 1,\ldots,n\}$ with $|I| =
l$, such that the mapping
\[
 M_{I,x}: {\mathfrak{h}} \ni A \mapsto ((Ax)_i)_{i \in I} \in \RR^I
\]
is a vector space isomorphism. This is the case iff ${\rm
det}(M_{I,x}) \not=0$. We let $P_I(x) = {\rm det}(M_{I,x})$, and
observe that since the determinant is polynomial, and $x$ enters
linearly in the definition of $M_{I,x}$, $P_I(x)$ is indeed a
polynomial.

It follows that the set of all $x$ for which (\ref{eqn:rank_cond})
holds is characterized by the condition~$P(x)\! \not=\! 0$, where
\[
 P(x) = \sum_{I \subset \{ 1,\ldots, n \}, |I|=l } P_I^2(x)~~.
\]
This set is clearly open, and if it is not empty, its complement has
measure zero. This settles (b), and the algorithm for (c) is now
clear: Check whether all coefficients of $P$~vanish. \chhf{Since the degree of $P$ is $\le 2l$, this can be done in finitely many steps.}
\end{proof}

As a further necessary criterion, we note:
\begin{cor}
Let $H < GL(n,\RR)$ be an admissible abelian matrix group, $H \cong
\RR^l \times \ZZ^m \times \TT^k \times F$. Then $l+m \le n$.
\end{cor}

%\newpage

\ch{\begin{thm}\label{thm:adm sca iff}
 Let $H < {\rm GL}(n,\mathbb{R})$ be a closed simply connected abelian matrix group \cch{of dimension $k$} with positive spectrum. Then $H$ is
 admissible if and only if
\begin{enumerate}
 \item[(i)] There exists $h \in H$ with $|{\rm det}(h)| \not= 1$;
 \item[(ii)] For some $x \in \RR^n$ (and hence for almost all) the linear map
\begin{equation} \label{eqn:rank_cond_2nd}
  {\mathfrak{h}} \ni A \mapsto xA \in \RR^n
\end{equation}
 has rank $k$.
\end{enumerate}
\end{thm}}

 \begin{proof}

\ch{Condition (i) is necessary by Theorem \ref{thm:char_adm}, and
necessity of condition (ii) has been shown at the beginning of the
section. For the sufficiency we note first that by the above
considerations condition (ii) holds if and only if  there exists at
least one (and hence almost all) dual orbit on which $H$ acts
locally freely. We will proceed by induction over the number of
blocks in the joint tridiagonalization of the infinitesimal
generators of $H$.}

 The following theorem, which is \chhf{mainly} due to Chevalley and Rosenlicht, is essential for treating the single block case.
 Note that the theorem yields a measurable set meeting {\em every} orbit,
 not just almost every orbit, in  a single point.
 \begin{thm} \label{thm:orbit_unipt}
Let $H$ be a closed connected subgroup of $T(n,\RR)$, the group of
 upper triangular matrices with ones on the diagonal. Then there
 exists a measurable subset of~$\RR^n$ meeting each $H$-orbit in
 precisely one point.
 Moreover, $H$ acts freely on each orbit on which it acts
 locally freely.
 \end{thm}

 \begin{proof}
We rely on a theorem relating the existence of measurable
 fundamental domains to regularity of the orbits. More precisely,
 suppose a locally compact group $H$ acts continuously on the
 locally compact space $\RR^n$. Then by a result of Effros \ch{
 \cite[Theorems~2.1\nobreakdash--2.9, $(2) \Leftrightarrow (12)$]{Ef}},  there exists a
 measurable set meeting each orbit in $\RR^n$ precisely once iff
 for every $x \in U$ the natural mapping $H \in h \mapsto h.x$
 induces a homeomorphism $H/{\rm stab}_H(x) \to H.x$. But the Chevalley-Rosenlicht theorem
 \cite[Theorem 3.1.4]{CoGr} shows precisely that, with the additional information that the $H$-orbits are simply
 connected.

 The second fact then implies that $H$ acts freely whenever it acts locally freely: If ${\rm stab}_H(x)$ is discrete, it follows that
 ${\rm stab}_H(x) \cong \mathbb{Z}^m$, and then $H/{\rm stab}_H(x) \cong \mathbb{R}^{n-m} \times \mathbb{T}^m$. But the quotient is simply connected, which implies that
 $m=0$.
 \end{proof}

 The following lemma is needed for the induction step.
 \begin{lemma} \label{lem:mfd_ind}
 Let $H_i < GL(V_i)$ be given, where $V_i$ are vector spaces. Let
 $\sigma: H_1 \to GL(V_2)$ be a group homomorphism with
 $\sigma(h_1) h_2 = h_2 \sigma(h_1)$, for all $h_1 \in H_1, h_2 \in
 H_2$. Let $H < GL(V_1 \oplus V_2)$ be given by
 \[ H = \{ h = (h_1,\sigma(h_1) h_2): h_1 \in H_1, h_2 \in H_2 \}
 ~~. \] Let $S_i \subset V_i$ denote a measurable fundamental
 domain of $V_i/H_i$, with the additional property that $H_i$ acts freely
 on the orbits running through $S_i$. Then $S_1 \times S_2$ is a measurable
 fundamental domain for $(V_1 \oplus V_2) / H$, and $H$ acts freely on the orbits
 running through $S$.
 \end{lemma}
 \begin{proof}
 We compute
 \[ U = H.(S_1 \times S_2) = \{ (h_1 .s_1, \sigma(h_1) h_2 . s_2) : h_i
 \in H_i, s_i \in S_i \} ~~.\] Then for almost every $y \in V_1$,
 $y = h_1.s_i$ for suitable $h_1 \in H_1$ and $s_1 \in S_1$.
 Moreover, for each such $y$, the slice
 \[ U \cap \{ y \} \times V_2 \] contains the elements $\sigma(h_1)
 h_2.s_2$, where $h_2 \in H_2$ and $s_2$ in $S$. Since
 $\sigma(h_1)$ is a fixed invertible linear mapping, the slice has
 complement of measure zero in $ \{ y \} \times V_2$ (in the
 Lebesgue measure on the affine subspace). Hence Fubini's theorem
 implies $|V_1 \times V_2 \setminus U| \!=\! 0$.

 For the second property of a measurable fundamental domain, assume $(s_1,s_2) \!=\! h.(\tilde{s}_1,\!\tilde{s}_2)
\! =\!\! (h_1 \tilde{s}_1,\sigma(h_1) h_2. \tilde{s}_2)$, for suitable $s_i,\tilde{s}_i \in S_i$ and $h=(h_1,h_2) \in H$.
 Then the fact that $S_1$ is a fundamental domain yields that
 $s_1 = \tilde{s}_1$, and $h_1$ is contained in the stabilizer. By the freeness assumption, $h_1= {\bf 1}$.
 Hence \ch{$h_2 \tilde{s}_2 = s_2 $}, resulting again in $s_2 = \tilde{s}_2$ and $h_2 = {\bf
 1}$.

We remark that this argument also covers the special case that
  $H_2$ is trivial; in this case, $S_2 = V_2$.
 \end{proof}

 Now let $A_1,\ldots,A_k$ denote a basis of the Lie algebra of $H$. We
 assume that $H$ has positive
  eigenvalues, hence all eigenvalues of the $A_i$ are real. Therefore there exists
  a basis of $\RR^n$ with respect to which the $A_i$ have the block diagonal form from
  \ch{Theorem \ref{thm:structure_commuting_real}}. The proof now proceeds by induction over the number of blocks
  in the decomposition.

  First suppose that the number of blocks equals one, i.e. each basis element
  has a single eigenvalue. If this eigenvalue is zero, for all $A_i$, then $H < T(n,\RR)$,
  and Theorem \ref{thm:orbit_unipt} applies. In the remaining case, we may assume after
  replacing some of the $A_i$ by suitable linear combinations, that $A_1$ has
  eigenvalue $1$, and the remaining $A_i$ have eigenvalue $0$. Hence $A_1 = \mathbf{1} + N$, with a suitable proper upper tridiagonal matrix
  $N$.

  Hence $H = H_1 H_2$, where $H_1 = {\rm exp}(\RR A_1)$, and $H_2 = {\rm exp}({\rm span} (A_i:i=2,\ldots,k))$.
  Note that $H_2 < T(n,\RR)$. Hence \ch{by theorem
  \ref{thm:orbit_unipt},}
  there exists a measurable set $S_2$ meeting each $H_2$-orbit
  in precisely one point. Observe that for all $h_2 \in H_2$ and all $x=(x_1,\ldots,x_n)^t \in \RR^n$,
  $h_2(x) = (y,x_n)$, for suitable $y \in \RR^{n-1}$. In particular, the
  affine subspaces $\RR^{n-1} \times \{ \pm 1 \}$ are $H_2$-invariant.
  Hence, $S=S_2 \cap (\RR^{n-1} \times \{ \pm 1 \})$ is
  a measurable set meeting each $H_2$-orbit in this $H_2$-invariant subset precisely
  once.

  We claim that $S$ meets each orbit in $ \RR^{n-1} \times (\RR \setminus \{ 0 \})$ in precisely one point.
  To see this, observe that $h_1 = {\rm exp}(r A_1)$ factors as $h_1(x) = e^r u(r)$, with $u(r)$ a unipotent matrix, keeping the $n$th coordinate of each vector
  fixed.

  Let $(y,x_{n})$ with $y \in \RR^{n-1}$ and $x_n \in \RR \setminus \{ 0 \}$ be given. Let
  $r = \log(|x_n|)$.  Then $\exp (-r A_1)(y,x_n) = (y',{\rm sign} \cch{(x_{n})})$, with suitable $y' \in \RR^{n-1}$.
  By choice of $S$ there exists $s \in S$ and unique $h \in H_2$ such that
  $(y,{\rm sign}\cch{(x_{n})}) = h.s$. This shows that each orbit is met.

  Moreover, if $h_1 h_2. s = \tilde{s}$ for $s, \tilde{s} \in S$ and $h_i \in H_i$, the comparison
  of the $n$th coordinates shows that $h_1$ is the identity. But $h_2 .s = \tilde{s}$ implies
  $s = \tilde{s}$, because of $S \subset S_2$.

  This shows that there exists a measurable set meeting each orbit in $\RR^{n-1} \times (\RR \setminus \{ 0
   \})$ precisely once. The action is free wherever it is locally free: For $H_2$, this observation was part
  of Theorem \ref{thm:orbit_unipt}, for $H_1 H_2$ it follows from this and the fact that $H_1$ acts freely
  on the $n$th variable. Hence we may replace \ch{$S$} by the smaller set $S' \subset S$ of all points
  in $S$ fulfilling in addition the local freeness condition (\ref{eqn:rank_cond}). This condition
  determines a measurable subset $S'$ of $S$, and the orbits going through this subset yield the
  intersection of $\RR^{n-1} \times (\RR \setminus \{ 0 \})$ with the set of all orbits of maximal
  dimension. \ch{By Proposition \ref{prop:local_free} and the hypothesis of the theorem,} this is a Borel set with complement of measure zero. Hence we have finally produced a
  measurable fundamental domain $S'$ with the property that $H$ acts freely on the orbits through $S'$.
 This concludes the one block case.

Now suppose we have shown the statement for $\ell$ blocks, and
  assume that the joint block diagonalization of the Lie algebra basis for $H$ has $\ell+1$
  blocks.
  Let $k'$ denote the dimension of the space spanned by the matrices
  $\tilde{A}_1,\ldots,\tilde{A}_k$, obtained by taking the first $\ell$
  blocks of $A_1,\ldots,A_k$. Denote by $n' < n$ the sum of the sizes of the first $\ell$
  blocks.

Then, passing to suitable linear combinations and reindexing
  allows the assumption \linebreak that $\tilde{A}_{k'+1}=\ldots=\tilde{A}_k=0$.
  We let $\ch{H_1^s }= {\rm exp}({\rm span}(\tilde{A}_1,\ldots,\tilde{A}_{k'})) \subset {\rm GL}(n',\RR)$,
   $ \ch{H_2^s }\subset {\rm GL}(n-n',\RR)$ the complementary
  subgroup corresponding to the remaining basis elements, $\ch{H_1^b\! =\! {\rm exp}({\rm span}(A_1,\ldots,A_{k'})) \!\subset\! {\rm GL}(n,\RR)}$,
  and $ \ch{H_2^b ={\rm exp}({\rm span}(A_{k'+1},\ldots,A_{k})) \subset {\rm GL}(n,\RR)}.$
 Then there exists a group homomorphism $\sigma :\ch{ H_1^s }\to {\rm
 GL}(n-n',\RR)$ such
  that
  \begin{equation} \label{eqn:struct_H}
 \ch{ H_1^{b} }= \left\{ \left( \begin{array}{cc} h_1 & 0 \\ 0 & \sigma(h_1) \end{array}
  \right) ~:~\ch{h_1 \in H_1^s} \right\}~~,
  \end{equation}
  and
  \begin{equation} \label{eqn:struct_H_2}
  H = \ch{H_1^b H_2^b} = \left\{ \left( \begin{array}{cc} h_1 & 0 \\ 0 & \sigma(h_1) h_2\end{array}
  \right) ~:~\ch{h_1 \in H_1^s, h_2 \in H_2^s} \right\}~~.
  \end{equation}
  The induction hypothesis yields measurable fundamental domains
  $S_1$ and $S_2$
  for \ch{$\mathbb{R}^{n'}/H_1^{s}$} and \ch{$\mathbb{R}^{n-n'}/H_2^{s}$}, respectively. \ch{Then} by Lemma
  \ref{lem:mfd_ind}, $\ch{S}=S_1 \times S_2$ is \ch{a} measurable fundamental
  domain for \ch{$\mathbb{R}^{n}/H$ and $H$ acts freely on all orbits running through
  $S$. Finally Theorem \ref{thm:char_adm} gives the desired conclusion.}
 \end{proof}

%\newpage

%\bf{Generalizations}\\

It turns out that the much more general solvable case can be treated
as well, using a moderate amount of Lie theory. For the pertinent
notions regarding general Lie groups and algebras we refer to
\cite{Var}; the results specific to exponential Lie groups can be
found in \cite{Bernatetal}.

\begin{defn}
 Let $\mathfrak{h}< {\rm gl}(n,\mathbb{R})$ be a Lie-subalgebra. We call $\mathfrak{h}$ {\em exponential} if for all $X \in \mathfrak{h}$, ${\rm ad}(X):\mathfrak{h} \to \mathfrak{h}$ does not have a purely imaginary eigenvalue.
\end{defn}

\begin{rem}
 It is known that all exponential Lie algebras are solvable. Given such a Lie algebra $\mathfrak{h}$, let $\widetilde{H}$ denote the associated connected, simply connected Lie group, and let $\widetilde{\exp}: \mathfrak{h} \to \widetilde{H}$ denote the exponential map.  Then $\mathfrak{h}$ is exponential iff $\widetilde{\exp}$ is a diffeomorphism \cite{Bernatetal}.
\end{rem}

The following well-known result is a special case of \cite[Theorem 3.7.3]{Var}.
\begin{lemma}
 Let $\mathfrak{h}<{\rm gl}(n,\mathbb{R})$ be solvable, and $(\varrho,V)$ an $\mathfrak{h}$-module. Then there exist $\mathbb{R}$-linear mappings $\lambda_1,\ldots,\lambda_d : \mathfrak{h} \to \mathbb{C}$ and a suitable basis $v_1,\ldots,v_d$ of $V$ such that, in the coordinates induced by $v_1,\ldots,v_d$,
\begin{equation}
 \label{eqn:tridiag_solv}
 \rho(X) = \left( \begin{array}{cccc} \lambda_1(X) &  & & \ast \\   & \lambda_2(X) &  & \\   &  & \ddots & \\ 0  &  &  & \lambda_d(X) \end{array} \right)~.
\end{equation}
The $\lambda_i$ are called {\em roots} of $V$.
\end{lemma}

The following crucial definition is taken from \cite{Bernatetal}.
\begin{defn}
 Let $\mathfrak{h}< {\rm gl}(n,\mathbb{R})$ be exponential, and $(\rho,V)$ an $\mathfrak{h}$-module. We call $(\rho,V)$ a
module of exponential type, if all roots of $V$ are of the type $\lambda(X) = \psi(X) (1+i\alpha)$, with a suitable linear functional $\psi:\mathfrak{h} \to \mathbb{R}$.
\end{defn}

We can now formulate the central result concerning exponential solvable Lie groups.
\begin{thm} \label{thm:solv_adm}
 Let $H< {\rm GL}(n,\mathbb{R})$ be a closed connected subgroup, with exponential Lie algebra $\mathfrak{h}$.
 Assume further that $\mathbb{R}^n$ is an $\mathfrak{h}$-module of exponential type (with respect to the natural
 action). \cch{Then :}
\begin{enumerate}
 \item[(a)] $H$ is simply connected.
 \item[(b)] For all $x \in \mathbb{R}^n$: If ${\rm stab}_H(x)$ is discrete, it is trivial.
 \item[(c)] There exists a Borel fundamental domain $C \subset \widehat{\mathbb{R}^n}$ for all dual orbits.
\end{enumerate}
\end{thm}
\begin{proof}
 First note that $\mathbb{R}^n$ and
$\widehat{\mathbb{R}^n}$ have the same roots, hence $\widehat{\mathbb{R}^n}$ is also of exponential type.

For (a) consider the left action of $\mathfrak{h}$ and $H$ on ${\rm gl}(n,\mathbb{R})$. It is equivalent to the diagonal action on an $n$-fold copy of $\mathbb{R}^n$, and therefore a $\mathfrak{h}$-module of exponential type as well. Furthermore, the stabilizers ${\rm stab}_H(A)$ are trivial, whenever $A$ is invertible. By \cite[Theorem~3.2.7]{Var}, the action lifts to an action of the simply connected covering group  $\widetilde{H}$.
Now \cite[Theorem I.3.3]{Bernatetal} implies that all associated stabilizers in $\widetilde{H}$ are connected. For invertible~$A$, the stabilizer coincides with the kernel of the covering map $p: \widetilde{H} \to H$. Since this kernel is also discrete, it has to be trivial. Thus $H$ is simply connected. Now part~(b) also follows from \cite[Theorem I.3.3]{Bernatetal}.

For the proof of part (c),  \cite[Theorem I.3.8]{Bernatetal} yields
that every dual orbit is open in its closure. Now the desired
statement follows from \ch{\cite[Theorems 2.6-2.9, $(5)
\Leftrightarrow (12)$]{Ef}}.
\end{proof}

\ch{Using Theorem \ref{thm:solv_adm}, Proposition
\ref{prop:local_free} and Theorem \ref{thm:char_adm} one obtains:}

\begin{cor} \label{cor:solv_adm}
  Let $H< {\rm GL}(n,\mathbb{R})$ be a closed connected subgroup, with exponential Lie algebra $\mathfrak{h}$. Then $H$ is admissible, provided that
\begin{enumerate}
 \item[(i)] There exists $h \in H$ with $|{\rm det}(h)| \not= \Delta_H(h)$;
 \item[(ii)] there exists at least one dual orbit on which $H$ acts locally freely;
 \item[(iii)] $\mathbb{R}^n$  is an $\mathfrak{h}$-module of exponential type.
\end{enumerate}

\ch{In particular, if $H$ is a closed, connected, abelian matrix
group with real spectrum, then $H$ is admissible if and only if (i)
and (ii) hold.}
\end{cor}

\section{\ch{Finitely generated abelian matrix groups with real
spectrum}}

\label{sect:fga_real_spect}

Consider the following situation: Suppose we are given a finite set
of invertible, pairwise commuting matrices $A_1,\ldots,A_k \in {\rm
GL}\ch{(n,\RR)}$. Our task is to decide whether there exists a
continuous wavelet transform associated to the matrix group
$\cch{H=H_{d}=} \langle A_1,\ldots, A_k \rangle$, and possibly to
give a description of the admissible vectors. In effect, we would
like to consider the wavelet system $(T_x D_{A^m} \psi)_{m \in
\ZZ^k, x \in \ch{\mathbb{R}^{n}}}$, where we use the notation $A^m =
A_1^{m_1} \ldots A_k^{m_k}$, and look for conditions to reconstruct
arbitrary $f \in {\rm L}^2(\RR^n)$ from its scalar products from the
wavelet system via the inversion formula
\begin{equation}
 f = \sum_{m \in \ZZ^k} \int_{\ch{\RR^n}} \langle f, T_x D_{A^m} \psi
 \rangle ~  T_x D_{A^m} \psi~dx~~,
\end{equation}
which is equivalent to requiring that $\psi$ satisfy the discrete Calder\'on
condition
\begin{equation} \label{eqn:cald_gen}
 \sum_{m \in \ZZ^k} |\cch{\widehat{\psi}}(A^{-m}\omega)|^2 = 1~~, (a.e.).
\end{equation}

Several obstacles present themselves, in the following natural
ordering:
\begin{enumerate}
\item The mapping $\ZZ^k \ni m \mapsto A^m \in H$ need not be
injective (it is onto by construction of $H$). That is, we need to
check whether the generators $A_1,\ldots,A_k$ are {\em free}
generators. \item $H$ need not be discrete. \item Is $H$
admissible? As we have seen above, this amounts to the existence
of a group element with determinant $\not=1$ (a property that only
needs to be checked on the generators), and the existence of a
measurable fundamental domain. The latter problem is quite hard to
access directly.
\end{enumerate}

\ch{If we assume that all $A_i$ have positive eigenvalues, we can
now decide all these questions by standard linear algebra
techniques. We can compute a connected group $H_c$ containing~$H_d$
cocompactly in a straightforward
 manner: First block diagonalization of the generators, then computation of their
 matrix logarithms, $B_i = {\rm log}A_i$, for $i=1,\ldots,k$. Then, letting ${\mathfrak{h}}_c = {\rm span}(B_1,\ldots,\cch{B_k})$,
 the exponential map is a diffeomorphism onto the  closed connected
 simply connected group $H_c \supset H_d$.
 Its restriction yields a group isomorphism $\langle B_1,\ldots, \cch{B_k} \rangle \to H_d$,
 where the left-hand side is understood as additive subgroup of the vector space ${\mathfrak{h}}_c$.
 The fact that $\langle B_1,\ldots, \cch{B_k} \rangle$ generates ${\mathfrak{h}}_c$ as a vector space
 implies that ${\mathfrak{h}}_c/\langle B_1,\ldots, \cch{B_k} \rangle$ is compact, and hence
 $H_c/H_d$ is compact. Moreover, the problem of deciding whether $H_d$ is closed (i.e., discrete), is transferred to the
 analogous properties of $\langle B_1,\ldots, \cch{B_k} \rangle$.}
 %where it boils down to checking linear independence.

%First, we block diagonalize the $A_i$. Then we compute $B_i =
%\log(A_i)$, for $i=1,\ldots,k$, and let ${\mathfrak{h}}_c = {\rm
%span}(B_1,\ldots,B_k)$. Let $H_c = \exp(\mathfrak{h}_c)$, then $H_c$
%is a closed connected and simply connected matrix group.
Now the
above list of questions can be answered in the following way:
\begin{enumerate}
\item The generators are free iff $B_1,\ldots,B_k \subset
{\mathfrak{h}}_c$ are free, with respect to the additive group of
that vector space. The latter means that any linear combination with integer coefficients, at least one of them nonzero, of $B_1,\ldots,B_k$ is nonzero. Clearly, this is the same as linear independence over the rationals, which can be checked by the Gauss algorithm.
\item Assuming that $B_1,\ldots,B_k$ are free, they
generate a subgroup algebraically isomorphic to $\ZZ^k$. This
subgroup is closed in ${\mathfrak{h}}_c$ iff ${\rm
dim}({\mathfrak{h}}_c) = k$: Note that since the $B_i$~span~${\mathfrak{h}}_c$, we always have ${\rm dim}({\mathfrak{h}}_c)
\le k$, and if "$<$" holds, the countable subgroup cannot be
closed by \cite[9.12]{HeRo}. Conversely, if the vectors
$B_1,\ldots,B_k$ span ${\mathfrak{h}}_c$ as a vector space, their
$\ZZ$-linear combinations are discrete.

By the exponential map, this statement transfers to
$A_1,\ldots,A_k$. In short, $H$ is discrete iff $(\log
A_1,\ldots,\log A_k)$ are $\mathbb{R}$-linearly independent.
\item Clearly, the determinant function is $\equiv 1$ on $\langle A_1,\ldots,A_k \rangle$ iff $|\det(A_i)| = 1$ for all $1 \le i \le k$.
\item With all previous tests passed, admissibility of $H$ is decided by checking local freeness of the action of $H_c$, using
\ref{prop:local_free} (c).
\end{enumerate}

 \ch{As a result of this discussion, we find a quite striking
 phenomenon as we pass from a connected group to a cocompact
 discrete subgroup. Recall that in principle there are two
 obstacles to  admissibility: Noncompact stabilizers and badly
 behaved orbit spaces. Theorem \ref{thm:adm sca iff} points out that
 for connected groups with positive eigenvalues only the
 stabilizers condition can fail. By contrast, we know that in the
 discrete case this condition is trivially fulfilled, confer Lemma
 \ref{lem:cp_stab}. Hence, as we pass from a connected group
 $H_c$ to a discrete cocompact subgroup $H_d$, one obstacle turns
 into the other, i.e., discretization of the noncompact stabilizers
 in $H_c$ results in pathological orbit spaces for $H_d$.}

\section{The general case: Partial results and counterexamples}

\label{sect:counter_ex}

For the general case the questions related to admissibility are not
so easily answered. In this section we present some counterexamples
showing that the simple \chhf{sufficient criteria that apply in  the
real spectrum case} are not valid in the general case. More
precisely, we consider connected matrix groups $H =
\exp(\mathfrak{h})$ and \ch{abelian matrix algebras containing~$H$}.
\chhf{The unit groups of the latter will allow to \cch{construct
counterexamples} showing that the sufficient criteria for the real
spectrum case, as formulated in Corollary \ref{cor:solv_adm}, no
longer work in the general case.}

The following theorem provides criteria for the existence of a measurable fundamental domain.
Note the gap between the necessary and the sufficient condition.
\begin{thm} \label{thm:char_mfd_algebra}
Let $H < {\rm GL}(n,\mathbb{R})$ be an abelian matrix group, and let $\mathcal{A} \subset gl(n,\mathbb{R})$ be an abelian matrix algebra with $H \subset \mathcal{A}$.
\begin{enumerate}
 \item[(a)] Suppose that there exists an $H$-invariant open subset $O \subset \mathbb{R}^n$
 of full measure such that, for all $x \in O$: ${\rm stab}_{\mathcal{A}^\times}(x) \cdot H$ is closed,
 \ch{where $\mathcal{A}^\times$ is the unit group of $\mathcal{A}$}.
  Then there exists a measurable fundamental domain for the $H$-orbits.
 \item[(b)] Suppose that there exists an $H$-invariant open subset $O \subset \mathbb{R}^n$ of positive measure such that, for all $x \in O$: ${\rm stab}_{\mathcal{A}^\times}(x) \cdot H$ is not closed. Then no measurable fundamental domain of the $H$-orbits in $\mathbb{R}^n$ can exist.
\end{enumerate}
\end{thm}
\begin{proof}
First note for the unit group $\mathcal{A}^\times$ that $\mathcal{A}^\times = \mathcal{A} \cap {\rm GL}(n,\mathbb{R})$, hence $\mathcal{A}^\times$ is an algebraic subgroup of ${\rm GL}(n,\mathbb{R})$. Hence, by \cite[3.1.3]{Zi}, the $\mathcal{A}^\times$-orbits in $\mathbb{R}^n$ are locally closed, and the natural projection $\mathcal{A}^\times \to \mathcal{A}^\times . x$ induces a homeomorphism $\mathcal{A}^\times / {\rm stab}_{\mathcal{A}}^\times(x) \to \mathcal{A}^\times . x$.

Now let us prove part (a): Since $H \subset \mathcal{A}^\times$, and
the larger group is abelian, $H$  acts on each orbit
$\mathcal{A}^\times. x \subset \mathbb{R}^n$. Let $p_x :
\mathcal{A}^{\times} \to \mathcal{A}^\times \cdot x$ denote the
quotient map. Then  $p_x^{-1}(H.x) = {\rm
stab}_{\mathcal{A}^\times}(x) \cdot H$ is closed, and thus $H. x
\subset \mathcal{A}^\times .x$ is relatively closed. Since the
latter is locally closed, $H.x$ is also locally closed, for all $x
\in O$. Now \ch{\cite[Theorems 2.6\nobreakdash--2.9, $(5) \Leftrightarrow
(12)$]{Ef}} yields the existence of a Borel set meeting \ch{each}
$H$-orbit in \ch{$O$} precisely once.

For the proof of part (b), let $G_x$ denote the closure of $H \cdot
{\rm stab}_{\mathcal{A}^\times}(x)$. Then $G_x \subset
\mathcal{A}^\times$ is a closed subgroup containing $H$. Let $\mu_x$
denote the Haar measure on $G_x /{\rm stab}_{G_x}(x)$, transferred
to the orbit $G_x .x$ via the quotient map. Then, by definition of
$G_x$, \ch{$H.y \subset G_x .x$} is dense for every $y \in G_x. x$,
which implies that $H$ acts ergodically on $G_x \cdot x$, if the
latter is endowed with $\mu_x$. Further, let $\nu_x$ denote the Haar
measure of $\mathcal{A}^\times /{\rm stab}_{\mathcal{A}^\times}(x)$,
transferred to $\mathcal{A}^\times .x$ by the quotient map. Then,
since the $\mathcal{A}^\times$-orbits are locally closed, it follows
that Lebesgue measure on \ch{$O$} decomposes into measures
equivalent to the $\nu_x$ \ch{(see \cite{Fu_Calderon})}. On the
other hand, each $\nu_x$ decomposes into the $\{ \mu_y : y \in
\mathcal{A}^\times .x \}$. But these are $H$-ergodic. Thus we have
found an ergodic decomposition into measures that are not supported
on single orbits. Now uniqueness of the ergodic decomposition yields
that Lebesgue measure on \ch{$O$} cannot be decomposed into measures
over the $H$-orbits in \ch{$O$}. Now \cite[Theorem~12]{Fu_Calderon}
\ch{and \cite[Theorem 2.65]{MiroThesis}} yields the desired
contradiction.
\end{proof}

We begin with the problem of deciding whether $H$ acts freely. Recall that, if the spectrum of the group is real, we only need to check for local a.e. freeness. The following example shows that in general this does not imply that the action is free a.e.:
\begin{ex}We construct a simply connected, connected closed abelian matrix group that acts locally freely but not freely almost everywhere. Let $\mathfrak{h} \subset gl(3,\mathbb{C}) \subset gl(6,\mathbb{R})$ be given as $\mathfrak{h} = {\rm span}(Y_1,Y_2,Y_3)$, where
\[
 Y_1 = \left( \begin{array}{ccc} i & 1 &   \\  & i & \\ & & i \end{array} \right)~,~ Y_2 =  \left( \begin{array}{ccc} 0 &  & 1  \\  & 0 & \\ & & 0 \end{array} \right), Y_3 =  \left( \begin{array}{ccc} 0 &  & i  \\  & 0 & \\ & & 0 \end{array} \right)~.
\] Note that only diagonal elements and nonzero off-diagonal elements are given, the remaining entries are zero. If $\mathfrak{h} = \mathfrak{h}_0 + \mathfrak{h}_1$ denotes the decomposition from Lemma \ref{lem:def_h0}, we see that $\mathfrak{h}_0$ is trivial, hence $H = \exp(\mathfrak{h})$ is a closed, simply connected abelian matrix group, with $\exp: \mathfrak{h} \to H$ a bijection. For $v \in \mathbb{C}^3 \equiv \mathbb{R}^6$ with $v_3 \not=0$, one immediately checks that the linear mapping
\[
 \mathbb{R}^3 \ni r \mapsto \sum_{i=1}^3 r_i Y_i v
\] is one-to-one, and thus the action of $H$ on the orbit of $v$ is locally free. On the other hand, introducing the vector $s \in \mathbb{R}^3$,
\[ s_1 = 1~,~ s_2 = - {\rm Re}\left( \frac{v_2}{v_3} \right)~,~s_3 = - {\rm Im}\left( \frac{v_2}{v_3} \right)~,
\] it follows for all $k \in \mathbb{Z}$ that $\exp\ch{(\sum_i 2 \pi k s_i Y_i)} (v) = v$. Thus ${\rm stab}_H(v)$ is not compact. \\
We can extend the group by including $Y_0 = \mathbf{1}_3$. The
resulting larger group \ch{$\widetilde{H}$} is a closed, simply
connected and connected abelian group. It fulfills the admissibility
criteria (i) and (iii) from Theorem \ref{thm:char_adm}, acts locally
freely, but nonetheless fails criterion (ii).  ((iii) is fulfilled
since \ch{$\mathbf{1}_3 \in \widetilde{\mathfrak{h}}$}. For checking
(i), use the algebra $\mathcal{A} = {\rm span}(Y_0,\ldots,Y_3)$
\ch{and apply Theorem \ref{thm:char_mfd_algebra} (a)}).
\end{ex}

Finally, an example of a simply connected closed matrix group
fulfilling all admissibility criteria \ch{of Theorem
\ref{thm:char_adm}} except
 (b)(i):
\begin{ex}
 Let $\mathfrak{h} \subset gl(4,\mathbb{C}) \subset \ch{gl(8,\mathbb{R})}$ be defined as $\mathfrak{h} = {\rm \span} \{Y_0,Y_1,Y_2,Y_3 \}$, where $Y_0 = \mathbf{1}_4$, and
 \ch{\begin{eqnarray*}
  Y_1 = \left( \begin{array}{cccc} i \pi & &  & \\ & i & 1 & \\ & & i &  \\ & & & i \end{array} \right) ~,~
Y_2 = \left( \begin{array}{cccc} 0 & &  & \\ & 0 & i & \\ & & 0 &
\\ & & & 0 \end{array} \right) ~,~
 Y_3 = \left( \begin{array}{cccc} 0  & &  & \\ & 0 &  & 1 \\ & & 0 &  \\ & & & 0 \end{array} \right)~.
 \end{eqnarray*}}%
Then it is easily checked that $H = \exp(\mathfrak{h})$ is a closed, simply connected abelian matrix group. Again, admissibility condition (iii) is guaranteed by including $Y_0$ in $\mathfrak{h}$. Furthermore, it is straightforward to verify that for every $v \in \mathbb{C}^4$ such that $v_1 \not=0 \not=v_4$ and with $(v_3,iv_4)$ $\mathbb{R}$-linearly independent, the stabilizer of $v$ in $H$ is trivial. \\
In order to see that $H$ does not fulfill the admissibility condition (i), we introduce the commutative matrix algebra $\mathcal{A}$ consisting of all
\begin{equation} \label{eqn:gen_elem_B}
 B =  \left( \begin{array}{cccc} z_1 & &  & \\ & z_2 & z_3 & z_4 \\ & & z_2 &  \\ & & & z_2 \end{array} \right) \in gl(4,\mathbb{C})~,
\end{equation} with $z_1,\ldots,z_4 \in \mathbb{C}$ arbitrary. Given $v \in \mathbb{C}^4$ with $v_1 \not= 0 \not= v_4$, we compute that
\[
{\rm stab}_{\mathcal{A}^\times} (v) = \{B \mbox{ as in } (\ref{eqn:gen_elem_B})~:~z_1 =z_2 = 1 ~,~ z_3v_3 + z_4 v_4 = 0 \}~.
\]
We next show that, whenever $iv_3,v_4$ are linearly independent,
then $H \cdot {\rm stab}_{\mathcal{A}^\times} (v)$ is not closed
\ch{and then Theorem \ref{thm:char_mfd_algebra} (b) gives the
desired conclusion}. To this end, observe that $H \cdot {\rm
stab}_{\mathcal{A}^\times}(v) = \exp(\mathfrak{g})$, with
\begin{equation} \label{eqn:Lie_alg_stab}
 \mathfrak{g} = \mathfrak{h} + \{ B \mbox{ as in } (\ref{eqn:gen_elem_B})~:~ z_1 = z_2=0~,~ z_3 v_3 + z_4 v_4 = 0 \} ~.
\end{equation}
Next, we note that for  $iv_3,v_4$ linearly independent,
\[
 {\rm span}_{\mathbb{R}} \left( \{ (i,0),(0,1) \} \cup \{ (z_3,z_4) \in \mathbb{C}^2 : z_3 v_3 + z_4 v_4 = 0 \} \right) = \mathbb{C}^2 ~.
\] To see this, observe that our assumptions guarantee that the $\mathbb{R}$-linear map
\[
 \mathbb{C}^2 \ni (z_3,z_4) \mapsto \ch{z_3 v_3}+z_4 v_4
\] is injective on the $\mathbb{R}$-span of  $(i,0),(0,1)$. Hence this space has trivial intersection with the kernel of the linear map, and thus a simple dimension argument yields that the two spaces span all of $\mathbb{C}^2$.

But this implies that $\mathfrak{g}$ contains the matrix
\[
  X_0 =  \left( \begin{array}{cccc} i \pi & &  & \\ & i &  &  \\ & & i &  \\ & & & i \end{array} \right)
\] and in fact, all (complex) diagonal matrices in $\mathfrak{g}$ with purely imaginary entries are real multiples of $X_0$.
\chhf{Thus,  if $\mathfrak{g} = \mathfrak{g}_0 + \mathfrak{g}_1$ is the decomposition from Lemma \ref{lem:def_h0}, then $\mathfrak{g}_0 = \mathbb{R} X_0$. But then Theorem \ref{thm:char_closed_abel_part} and Lemma
\ref{lemma:char_closed_abel}}
%\ref{thm:char_closed_abel}
allow to conclude that $\exp(\mathfrak{g})$ is not closed.
\end{ex}

\section*{Acknowledgements}
HF would like to thank the Universitat Aut\'onoma de Barcelona,
Departament de Ma\-te\-m\`a\-ti\-ques, for its hospitality.

\vspace{1cm}

\section*{Departament  de  Matem\`{a}tiques,  Universitat  Aut\`{o}noma
de Barcelona, 08193 Bellaterra-Barcelona, Catalonia.}
\vspace{-5mm} \hspace{4mm}\emph{E-mail
address:}\hfill\texttt{bruna@mat.uab.cat}\hfill
\texttt{jcufi@mat.uab.cat}\hfill \texttt{marga@mat.uab.cat}

\section*{Lehrstuhl A f\"{u}r Mathematik, RWTH Aachen, 52056 Aachen, Germany.}
\vspace{-5mm}
\begin{center}
\emph{E-mail
address:}\hspace{4mm}\texttt{fuehr@MathA.rwth-aachen.de}
\end{center}

\end{document}